\documentclass[a4paper,10pt]{article}
\usepackage{cmap}
\usepackage[T1,T2A]{fontenc}
\usepackage[utf8]{inputenc}
\usepackage[english]{babel}

\AtBeginDocument{}

\usepackage{tocloft}

\addto\captionsenglish{%
}

\usepackage[english]{babel}
\usepackage{amsfonts}
\usepackage{amssymb, amsthm}
\usepackage{xcolor,graphicx}
\usepackage{marginnote}
\usepackage{amsmath}
\usepackage{a4wide}
\usepackage[affil-it]{authblk}

\usepackage{epstopdf}
\usepackage{titlesec}
\titlelabel{\thetitle.\quad}

\usepackage[labelsep=period]{caption}

\usepackage{enumitem}

\numberwithin{equation}{section}

\let\OLDthebibliography\thebibliography
\renewcommand\thebibliography[1]{
	\OLDthebibliography{#1}
	\setlength{\parskip}{1pt}
	\setlength{\itemsep}{1pt plus 0.3ex}
}

\usepackage[colorlinks=true,linktocpage,pdfpagelabels,
bookmarksnumbered,bookmarksopen]{hyperref}
\definecolor{ForestGreen}{rgb}{0.1,0.6,0.05}
\definecolor{EgyptBlue}{rgb}{0.063,0.1,0.6}
\hypersetup{
	colorlinks=true,
	linkcolor=EgyptBlue,         
	citecolor=ForestGreen,
	urlcolor=olive
}

\usepackage[hyperpageref]{backref}

\allowdisplaybreaks
\usepackage{mathtools}
\mathtoolsset{showonlyrefs}

\usepackage{orcidlink}

\usepackage{accents}

\newtheorem{theorem}{Theorem}[section]
\newtheorem{lemma}[theorem]{Lemma}
\newtheorem{proposition}[theorem]{Proposition}

\theoremstyle{definition}
\newtheorem{remark}[theorem]{Remark}

\newcommand{\W}{W_0^{1,p}(\Omega)}
\newcommand{\Wso}{\widetilde{W}_0^{s,p}(\Omega)}
\newcommand{\intO}{\int_\Omega}

\title{A posteriori estimates for problems with monotone operators
	\\ \medskip}

\author{V.~E.~Bobkov \quad S.~E.~Pastukhova\\}
\date{}

\AtEndDocument{%
	\par
	\medskip
	\bigskip
	\begin{tabular}{@{}l@{}}%
		(V.~E.~Bobkov)\\[0.2em]
		\textsc{Institute of Mathematics, Ufa Federal Research Centre, RAS}\\
		\textsc{Chernyshevsky str.\ 112, 450008 Ufa, Russia}\\[0.3em]
		\orcidlink{0000-0002-4425-0218} 0000-0002-4425-0218\\[0.3em]
		\textit{E-mail address}: \texttt{bobkov@matem.anrb.ru, bobkovve@gmail.com}\\[1.5em]
		
		(S.~E.~Pastukhova)\\[0.2em]
		\textsc{MIREA -- Russian Technological University}\\
		\textsc{Vernadsky Avenue 78, 119454 Moscow, Russia}\\[0.3em]
		\textit{E-mail address}: \texttt{pas-se@yandex.ru} 
\end{tabular}}

\begin{document}

	\maketitle
	
	\vspace*{-5ex}
	\begin{abstract}
		We propose a method of obtaining a posteriori estimates which does not use the duality theory and which applies to variational inequalities with monotone operators, without assuming the potentiality of operators. 
		The effectiveness of the method is demonstrated on problems driven by nonlinear operators of the $p$-Laplacian type, including the anisotropic $p$-Laplacian, polyharmonic $p$-Laplacian, and fractional $p$-Laplacian.

		
	\end{abstract}
	
\renewcommand{\cftdot}{.}
\begin{quote}	
	\tableofcontents	
	\addtocontents{toc}{\vspace*{-2ex}}
\end{quote}
	
	\section{Concept of a posteriori estimates}\label{sec:general}
	
	Let $K$ be a convex closed set in a reflexive Banach space $X$.
	Fix some $h \in X^*$, where $X^*$ is the dual space.
	Consider the problem: find $u \in K$ such that 
	\begin{equation}\label{eq:abstrat1}
		\langle A u - h, w-u \rangle \geq 0 \quad \text{for any}~ w \in K. 
	\end{equation}
	Here $A: K \to X^*$ is a bounded coercive semicontinuous and strictly monotone operator.
	It is well known that the problem \eqref{eq:abstrat1} has the unique solution $u \in K$, see, e.g.,
	\cite[Chapter~2, Section~8]{lions}. 
	In the terminology of \cite{lions}, the problem \eqref{eq:abstrat1} is called a \textit{variational inequality}.
	It is easily seen that if $K = X$ or the solution $u$ of the problem \eqref{eq:abstrat1} is an interior point of $K$, then $u$ satisfies 
	 \begin{equation}\label{eq:abstrat0}
	 	\langle A u - h, w \rangle = 0 \quad \text{for any}~ w \in X.
	 \end{equation}
 
	The operator $A$ often arises as the Gateaux derivative of some convex functional $J$ defined on $X$, i.e., there is a limit
	$$
	\lim\limits_{\lambda \to 0} \frac{1}{\lambda} (J(v+\lambda w) - J(v)) = \langle J'(v), w \rangle
	$$
	for $v, w \in X$ and $J'(v) = Av$. 
	In this case, $A$ is called a potential operator, and the problem \eqref{eq:abstrat1} is equivalent to the minimization problem on the set $K$ for the functional
	$$
	F(v) = J(v) - \langle h, v \rangle.
	$$
	In general, $A$ does not have to be potential and the inequality \eqref{eq:abstrat1} is not a problem of the calculus of variations.
 
 	\smallskip
 	Since the solution $u$ of the problem \eqref{eq:abstrat1} can rarely be found in a closed form, a natural question arises: given an arbitrary function $v \in K$, is it possible to evaluate some measure of its deviation from $u$ only in terms of $v$ and the data of the problem? 
 	Estimates of such type are called \textit{a posteriori estimates}, and a large amount of works are devoted to their finding and study, see, e.g., the overviews \cite{AO,RepBook1,RepBook2,Verf}. 
 	A posteriori estimates are relevant both in numerical calculations and in the study of the dependence of solutions on the data of the problem and its perturbations. 
 	
 	One of the main measures of deviation of an approximation $v$ from the solution $u$ is the norm $\|u-v\|_X$. 
 	In this case, any functional (majorant) $M: K \to \mathbb{R}$ such that
 	\begin{equation}\label{eq:apost-norm}
 		\|u-v\|_X \leq M(v), \quad v \in K, 
 	\end{equation}
 	can be considered as a posteriori estimate. 
 	The existence of such a majorant, especially in an explicit form, provides the possibility of estimating the quality of an experimentally obtained approximation of the solution.
 	At the same time, it is desirable to have a majorant satisfying the conditions
 	\begin{align*}
 		&M(v) = 0  \quad \iff \quad v = u,\\
 		&M(v) \to 0 \quad \text{as}~ v \to u ~\text{in}~ X.
 	\end{align*}
	Then the corresponding a posteriori estimate is called consistent.
	Note, however, that the choice of the norm of $X$ as a measure of deviation is not the only option and, probably, it is not always optimal even in the case of $K=X$, see, for example, \cite{BL,LY,Rep2000-2,Rep2023} and the following section.  
	In particular, in problems with $K \neq X$, the closeness of $v$ to $u$ only in the norm of $X$ might not indicate how $v$ is ``close'' to satisfy the inequality \eqref{eq:abstrat1}. 
	The investigation of quality of various deviation measures, as well as their comparison with each other, seems to be an independent interesting problem.

	In the present work, we propose a general method of obtaining a posteriori estimates of a ``functional'' type and demonstrate its effectiveness on specific examples with nonlinear operators. 
	This method is simple, transparent, not related to variational settings of the problem, does not require the potentiality of the operator, and is based only on the strict monotonicity of the operator.
	The natural price for the simplicity and generality of the method is a \textit{possible} nonoptimality of the obtained a posteriori estimates in problems equipped with additional structures and conditions (for example, such as explicit functional spaces, integral representations, the divergent form of the operator, additional regularity of the date, etc.). 
	For instance, methods of obtaining a posteriori estimates for variational problems based on the duality theory are developed in \cite{RepBook1,RepBook2} (see also references therein).
	The development of these ideas to problems with monotone operators is given in \cite{Rep2023}, but even in this case the operators require an additional structure allowing to work with dual objects, and the considered problems are of the variational nature. 
	At the same time, there are numerous model examples in nonlinear analysis that do not allow dual formulations.
	In any case, the method of the present work can be used to obtain a posteriori estimates for a wide range of problems, perhaps, with the prospect to improve them by employing other methods, including the duality theory.
	 	
	We also note that the so-called method of integral identities, considered, for example, in \cite [Chapter~3]{RepBook1} (see also references therein), can be interpreted as a special case of the method proposed in the present work.
	
	In the conclusion of this introductory section, we note that, along with a posteriori estimates, there are a priori estimates for the analysis of approximate solutions: in such estimates, the rate of convergence of an approximation $v$ to the exact solution $u$ is controlled in terms of $u$.
	A priori estimates of this type can be useful to justify the convergence of a sequence of approximations to the solution of a considered problem.
	For example, it was proved in \cite{PYa2,PY,ZY} that the rate of convergence of Galerkin approximations to the exact solution $u$ of the  Dirichlet problem with nonlinear operators of the $p$-Laplacian type is subordinated to the distance between $u$ and a finite-dimensional subspace $X_n$ determining the $n$-th Galerkin approximation, see also \cite{BL,CK,LY}. 
	Compared to this, a posteriori estimates allow to quantify the deviation of an arbitrary element $v \in X_n$ from $u$ in a suitable norm by a majorant which does not dependent on the solution $u$, the latter one being unknown explicitly.

	\section{General approach to variational inequalities}\label{sec:method}
	
	In order to find an a posteriori estimate \eqref{eq:apost-norm} for the solution $u$ of the problem \eqref{eq:abstrat1}, we take an arbitrary element $v \in K$ and introduce a certain analogue of the norm of $Av - h$ over the set $K$:
	\begin{equation}\label{eq:K*}
		\|Av - h\|_{*,K} 
		:= 
		\sup_{w \in K \setminus \{v\}} \frac{(-\langle Av - h, w - v \rangle)_+}{\|w - v\|_X}.
	\end{equation}
	Hereinafter, we use the notation $a_+ = \max\{a,0\}$.  
	By virtue of the inequality \eqref{eq:abstrat1}, we have $\|Au - h\|_{*,K} = 0$, and for any $v \in K \setminus \{u\}$ the value of \eqref{eq:K*} is positive.
	Indeed, taking $w=u$ in \eqref{eq:K*} and applying the identity
	\begin{equation}\label{eq:abstract:object}
		-\langle Av - h, u-v \rangle
		=
		\langle Au - Av, u-v \rangle
		+
		\langle Au - h, v-u \rangle,
		\quad v \in K, 
	\end{equation} 
	we derive $\|Av - h\|_{*,K} > 0$ for any $v \in K \setminus \{u\}$ from the strict monotonicity of the operator $A$ and the inequality \eqref{eq:abstrat1}.
	This fact also follows from the uniqueness of the solution of the problem \eqref{eq:abstrat1}: if $\|Av - h\|_{*,K} = 0$ for some $v \in K$, then
	$$
	\langle A v - h, w-v \rangle \geq 0 \quad \text{for any}~ w \in K,
	$$
	and hence $v=u$. 
	Note that if $A$ is defined on the whole $X$, then $\|Av - h\|_{*,K} \leq \|Av - h\|_{*}$, where $\|Av - h\|_{*}$ denotes the standard operator norm:
	$$
	\|Av - h\|_{*} 
	= 
	\sup_{w \in X \setminus \{v\}} \frac{\langle Av - h, w - v \rangle}{\|w - v\|_X}
	=
	\sup_{\omega \in X \setminus \{0\}} \frac{\langle Av - h, \omega \rangle}{\|\omega\|_X}.
	$$
	In the case $K=X$, we have $\|Av - h\|_{*,K} = \|Av - h\|_{*}$. 
	It is desirable to have the continuity property
	\begin{equation}\label{eq:convAA}
	\|Av - h\|_{*,K} \to \|Au - h\|_{*,K} = 0
	\quad \text{as}~ v \to u ~\text{in}~ X, ~~ \{v\} \subset K,
	\end{equation}
	which, generally speaking, is nontrivial and depends on properties of $A$ and other data of the problem, as well as on the choice of the sequence $\{v\}$. 
	If $K=X$ and $A$ is continuous in the strong topology of $X^*$, then \eqref{eq:convAA} is valid.
	
	\medskip
	In view of the notation \eqref{eq:K*}, for any $v \in K$ we have
	\begin{equation}\label{eq:abstract:up00}
	 -\langle Av - h, u-v \rangle \leq \|Av - h\|_{*,K} \, \|u-v\|_X.
	\end{equation}
	Using the identity \eqref{eq:abstract:object}, we rewrite the estimate \eqref{eq:abstract:up00} in the following more convenient form:
	\begin{equation}\label{eq:abstract-main00}
		\frac{\langle Au - Av, u-v \rangle}{\|u-v\|_X}
		+ 
		\frac{\langle Au - h, v-u \rangle}{\|u-v\|_X}
		\leq
		\|Av - h\|_{*,K},
	\end{equation}
	where the left-hand side is positive for any $v \in K \setminus \{u\}$.
	For consistency, we set the left-hand side of \eqref{eq:abstract-main00} to be zero when $v=u$. 
	If $A$ is continuous in the strong topology of $X^*$, then for the first term on the left-hand side of \eqref{eq:abstract-main00} we have
	$$
	\frac{\langle Au - Av, u-v \rangle}{\|u-v\|_X} \to 0
	\quad \text{for any sequence}~ v \to u ~\text{in}~ X,~~ \{v\} \subset K \setminus \{u\},
	$$
	and hence the term 
	$
	\frac{\langle Au - Av, u-v \rangle}{\|u-v\|_X}
	$ 
	serves as a natural measure of deviation of $v$ from $u$ in $X$. 
	The second term on the left-hand side of \eqref{eq:abstract-main00} no longer has such a property in the general case if $K \neq X$.
	However, it follows from the boundedness of $A$ that any sequence $v \to u$, $\{v\} \subset K \setminus \{u\}$, contains a subsequence (which we again denote as $\{v\}$) for which
	$$
	\frac{\langle Au - h, v-u \rangle}{\|u-v\|_X} \to C_{\{v\}}
	\in [0,+\infty), 
	$$
	where $C_{\{v\}}$ is a constant depending on $\{v\}$. 
	If $C_{\{v\}}>0$, then this indicates that the sequence $\{v\}$ poorly approximates the property of $u$ to be a solution of \eqref{eq:abstrat1}, since
	$$
	0 < C_{\{v\}} \leq \lim_{v \to u} \|Av - h\|_{*,K}. 
	$$
	The term $\frac{\langle Au - h, v-u \rangle}{\|u-v\|_X}$ can be especially relevant in the case $K \neq X$, see, e.g., Remark~\ref{rem:ANR} about an obstacle problem for the $p$-Laplacian. 
 	Note that if there is a convergence \eqref{eq:convAA} for some sequence $\{v\}$, then both sides of \eqref{eq:abstract-main00} tend to zero on this sequence.
 	
	Apparently, the estimate \eqref{eq:abstract-main00} is the most general estimate of a posteriori type that can be obtained under the imposed assumptions.
	
	Another form of a posteriori estimate similar to \eqref{eq:abstract-main00} can be obtained from \eqref{eq:abstract:up00} by extracting $\|u-v\|_X$ from the product $\|Av - h\|_{*,K} \, \|u-v\|_X$ via Young's inequality $ab \leq f(a) + f^*(b)$, where $f$ and $f^*$ are some convex Young conjugate functions:
	\begin{equation}\label{eq:abstract-main01}
		\langle Au - Av, u-v \rangle 
		+ 
		\langle Au - h, v-u \rangle
		- 
		f(\|u-v\|_X) 
		\leq
		f^*(\|Av - h\|_{*,K}).
	\end{equation}
	In this case, the choice of a suitable function $f$ should ensure the positivity of the left-hand side of \eqref{eq:abstract-main01} and it depends on properties of the operator $A$ and other data of the problem.
	
	\medskip
	Suppose now that the strict monotonicity of $A$ can be refined as follows:
	$$
\langle Au - Av, u-v \rangle \geq G(\|v\|_X, \|u-v\|_X), 
\quad v \in K,
$$
where the function $G : [0,+\infty)^2 \to [0,+\infty)$ is such that the mapping 
$t \mapsto G(s,t)/t$ is nondecreasing for $t \geq 0$, tends to $0$ as $t \to 0$, and invertible for any $s \geq 0$. 
Denote
$$
g(s,\cdot) := \left(t \mapsto \frac{G(s,t)}{t}\right)^{-1}.
$$
In this case, it follows from the estimate \eqref{eq:abstract-main00} that
\begin{equation}\label{eq:abstract-main3}
		\frac{G(\|v\|_X, \|u-v\|_X)}{\|u-v\|_X}
		+ 
		\frac{\langle Au - h, v-u \rangle}{\|u-v\|_X}
		\leq
		\|Av - h\|_{*,K}.
	\end{equation}
	In particular, noting the nonnegativity of the second term on the left-hand side of \eqref{eq:abstract-main3} and using the properties of the function $G$, we deduce the inequality
	\begin{equation}\label{eq:abstract-main1}
		\|u-v\|_X \leq g(\|v\|_X, \|Av - h\|_{*,K}).
	\end{equation}
	If we estimate the first term on the left-hand side of \eqref{eq:abstract-main3} by zero, then, using \eqref{eq:abstract-main1}, we obtain
	\begin{equation*}\label{eq:abstract-main2}
	\langle Au - h, v-u \rangle
	\leq 
	\|Av - h\|_{*,K} \,\, g(\|v\|_X, \|Av - h\|_{*,K}).
	\end{equation*}

\medskip
	The quantity $\|Av - h\|_{*,K}$ is not determined constructively. 
	For a number of specific model operators, mainly defined by integral forms, there are methods related to the structure of the operator and the regularity of the corresponding solutions, which allow to estimate $\|Av - h\|_{*,K}$ from above constructively and, in some cases, to guarantee its convergence to zero as $v \to u$ in $X$, see, e.g., \cite{RepBook1,RepBook2}.
	
	We demonstrate the application of the general method and the derivation of constructive estimates on $\|Av - h\|_{*,K}$ and the right-hand sides in \eqref{eq:abstract-main00}, \eqref{eq:abstract-main3}, \eqref{eq:abstract-main1} in the following sections.
	We consider the Poisson equation with Dirichlet boundary conditions (the Poisson problem, for brevity; Section~\ref{sec:p-Poisson}), an obstacle problem (Section~\ref{sec:obstacle}), the Poisson problem with an anisotropic operator (Section~\ref{sec:px-Poisson}), the Poisson equation with Neumann boundary conditions (Section~\ref{sec:Neumann}), the vector Poisson problem (Section~\ref{sec:vector}), an obstacle problem for a polyharmonic operator (Section~\ref{sec:polyharm}), and the Poisson problem with a nonlocal operator (Section~\ref{sec:fractional-Poisson}).
	All our examples are nonlinear and defined by operators of the $p$-Laplacian type.
		
	Constructive a posteriori estimates for the Poisson problem obtained in Section~\ref{sec:p-Poisson}, as well as related estimates for the obstacle problem considered in Section~\ref{sec:obstacle}, are known in the literature, but we derive them using a different method, which is explicitly noted in the corresponding remarks. 
	These results are given in the present work in order to demonstrate the method on more or less simple classical examples.
	At the same time, a posteriori estimates proposed in Sections \ref{sec:collection}, \ref{sec:polyharm}, \ref{sec:fractional-Poisson} seem to be new and, apparently, have not been presented in the literature in the considered generality.

	\medskip
	In what follows, unless explicitly stated otherwise, we assume that
	$p>1$, $\Omega \subset \mathbb{R}^N$ is a domain of bounded measure, 
	$N \geq 1$, and $h \in L^{p'}(\Omega) \setminus \{0\}$, where $p' := p/(p-1)$. 
	We denote by $\|\cdot\|_q$ the standard norm of $L^q(\Omega)$ and sometimes write
	$\|\cdot\|_{q,D}$ to emphasize that the domain of integration is a set $D$. 
	By $B_R$ we denote an open ball of radius $R$ in $\mathbb{R}^N$, by $a \cdot b$ the scalar product of vectors $a,b \in \mathbb{R}^k$, by $W^{m,p}(\Omega)$ and $W_0^{m,p}(\Omega)$ the standard Sobolev spaces, and by $\langle \cdot, \cdot \rangle$ the duality between $Y^*$ and $Y$ for any vector space $Y$.

	\section{Poisson problem}\label{sec:p-Poisson}

	Consider the problem of finding such a function $u \in \W$ that
	\begin{equation}\label{eq:weak1}
		\intO |\nabla u|^{p-2} \nabla u \cdot \nabla \xi \,dx
		=
		\intO h \xi \,dx
		\quad \text{for any}~ \xi \in \W.
	\end{equation}
	This problem coincides with \eqref{eq:abstrat0} by choosing $X = \W$, $X^* = W^{-1,p'}(\Omega)$, and $A = -\Delta_p : \W \mapsto W^{-1,p'}(\Omega)$. 
	For smooth functions, the $p$-Laplacian $\Delta_pu$ can be defined pointwisely as $\Delta_pu = \text{div}(|\nabla u|^{p-2} \nabla u)$.
	The $p$-Laplacian is potential, hence the problem \eqref{eq:weak1} can be interpreted as the problem of finding critical points of the functional
	\begin{equation}\label{eq:F}
	F(v) = \frac{1}{p}\intO |\nabla v|^p \,dx 
	- 
	\intO h v \,dx, \quad v \in \W.
	\end{equation}
	In particular, in the notation of Section~\ref{sec:method}, 
	\begin{equation*}\label{eq:AAF}
	\|Av - h\|_{*,\W} = \|Av - h\|_{*} = \|F'(v)\|_*, \quad v \in \W.
	\end{equation*}
	It is well known that the functional $F$ is Gateaux differentiable, coercive, and weakly lower semi-continuous. Moreover, $F$ is strictly convex in $\W$. 
	Therefore, $F$ has the unique critical point $u \in \W$, which is its global minimum and a weak solution to the problem
	\begin{equation*}\label{eq:Poisson:pointwise}
		\left\{
		\begin{aligned}
		-\Delta_p u &= h \quad \text{in}~ \Omega,\\
		u &= 0 \quad \text{on}~ \partial\Omega.
		\end{aligned}
		\right.
	\end{equation*}
	Below, we will use the following energy estimate obtained from \eqref{eq:weak1} with 
	$\xi = u$ and the H\"older and Friedrichs inequalities:
	\begin{equation}\label{eq:energy-id1}
		\|\nabla u\|_p^p 
		\leq 
		C_F^{p'} \|h\|_{p'}^{p'},
	\end{equation}
	where $C_F$ is the constant of embedding $\W \subset L^p(\Omega)$.
	Although $C_F$ is defined non-constructively, explicit two-sided bounds on it are well known.
	Since $C_F = \lambda_1^{-1/p}(p;\Omega)$, where $\lambda_1(p;\Omega)$ stands for the first eigenvalue of the $p$-Laplacian in $\Omega$ with Dirichlet boundary conditions, estimates on $C_F$ are expressed through estimates on $\lambda_1(p;\Omega)$.
	One of the classical lower bounds is the Faber--Krahn inequality:
	\begin{equation*}\label{eq:FK}
	\lambda_1(p;\Omega) \geq |\Omega|^{-p/N} |B_R|^{p/N} \lambda_1(p;B_R) \quad \text{for any } R>0.
	\end{equation*}
	The domain monotonicity of the first eigenvalue gives a simple upper bound
	$$
	\lambda_1(p;\Omega) \leq \lambda_1(p;B_{r_\Omega}), 
	$$
	where $r_\Omega$ is the inradius of $\Omega$. 
	In turn, the following constructive two-sided estimates on $\lambda_1(p;B_R)$ in terms of $p$, $N$, $R$ are given, e.g., in \cite{BD}:
	$$
	\frac{Np}{R^p} 
	\leq
	\lambda_1(p;B_R)
	\leq
	\frac{(p+1)(p+2)\cdots(p+N)}{N! R^p}.
	$$
	
	We introduce the following auxiliary space of the Sobolev type:
	\begin{equation*}\label{eq:Q*}
	Q^* = \{\eta \in L^{p'}(\Omega;\mathbb{R}^N):~ \text{div}\, \eta \in L^{p'}(\Omega)\}.
	\end{equation*}
	\begin{remark}[Regularity]\label{rem:regQ}
	It is well known that the assumption $h \in L^{p'}(\Omega)$ yields $u \in C(\Omega) \cap L^\infty(\Omega)$. 
	If $h \in L^{\infty}(\Omega)$, then $u \in C^{1,\alpha}(\Omega)$ for some $\alpha \in (0,1)$, see \cite{diben} and also \cite{lieberman}.
	On the other hand, since $u \in \W$, the flux $|\nabla u|^{p-2} \nabla u$ belongs to $L^{p'}(\Omega;\mathbb{R}^N)$. 
	The question of whether $|\nabla u|^{p-2} \nabla u$ belongs to $Q^*$ is nontrivial in the case $p \neq 2$. 
	An assumption guaranteeing that
	\begin{equation}\label{eq:pointwise}
	-\Delta_p u = h ~\text{a.e.\ in}~ \Omega
	\quad \text{and} \quad
	|\nabla u|^{p-2} \nabla u \in Q^*
	\end{equation}
	is a sufficient regularity of the flux, namely, $|\nabla u|^{p-2} \nabla u \in W^{1,1}_{\text{loc}}(\Omega; \mathbb{R}^N)$. This follows from the weak formulation \eqref{eq:weak1} with $\xi \in C_0^{\infty}(\Omega)$ and the fundamental lemma of the calculus of variations. 
	Such regularity of the flux takes place if $h \in L^{\max\{2,p'\}}(\Omega)$, see \cite{ciacni-mazja} and discussions in \cite{antonini,CM,guarnotta}. 
	\end{remark}
	
	Using the approach described in Section~\ref{sec:method}, we prove the following a posteriori estimate.
	\begin{theorem}\label{thm:1}
		Let $p \geq 2$. 
		Then for the solution $u$ of \eqref{eq:weak1} and any $v \in \W$ the following estimate holds:
		\begin{equation}\label{eq:thm:1}
			\|\nabla (u-v)\|_p
			\leq
			2^{\frac{p-2}{p-1}} \|F'(v)\|_{*}^\frac{1}{p-1}
			\leq 
			2^{\frac{p-2}{p-1}} 
			\inf_{\eta^* \in Q^*}\left(
			\|\tau^* - \eta^*\|_{p'} + C_F \|\mathrm{div}\, \eta^* + h\|_{p'}
			\right)^{\frac{1}{p-1}},
		\end{equation}
	where $\tau^* := |\nabla v|^{p-2} \nabla v$. 
	\end{theorem}	
	\begin{proof}
		Let us estimate the following expression from below and from above:
		\begin{align}
		-\langle F'(v), u-v \rangle
		&=
		-\intO |\nabla v|^{p-2} \nabla v \cdot \nabla (u-v) \,dx + \intO h (u-v) \,dx\\
		\label{eq:Fu-v}
		&=
		\intO (|\nabla u|^{p-2} \nabla u - |\nabla v|^{p-2} \nabla v) \cdot \nabla (u-v) \,dx.
		\end{align}
		By virtue of the well known algebraic inequality (see, e.g., \cite[Chapter~12, (I)]{Lind})
		\begin{equation}\label{eq:algebr-ineq1}
			(|b|^{p-2}b - |a|^{p-2}a) \cdot (b-a) \geq 2^{2-p} |b-a|^p,
			\quad a,b \in \mathbb{R}^N, \quad p \geq 2,
		\end{equation}
	the lower bound on \eqref{eq:Fu-v} takes the form
		\begin{equation}\label{eq:thm1:proof:0}
			-\langle F'(v), u-v \rangle
			\geq
			2^{2-p} \intO |\nabla (u - v)|^p \,dx
			=
			2^{2-p} \|\nabla (u - v)\|_p^p.
		\end{equation}
		On the other hand, we have the upper bound
		\begin{equation}\label{eq:thm1:proof:1x}
		-\langle F'(v), u-v \rangle \leq \|F'(v)\|_{*} \, \|\nabla (u-v)\|_p.
		\end{equation}
		Combining \eqref{eq:thm1:proof:0} with \eqref{eq:thm1:proof:1x}, we obtain the estimate of the type~\eqref{eq:abstract-main3}:
		\begin{equation}\label{eq:thm1:proof:2x}
		2^{2-p} \|\nabla (u - v)\|_p^{p-1} \leq \|F'(v)\|_{*}.
		\end{equation}
		In order to estimate $\|F'(v)\|_{*}$, we take arbitrary $\xi \in \W$ and $\eta^* \in Q^*$ and get
		\begin{align}
			|\langle F'(v), \xi \rangle|
			&=
			|\intO (|\nabla v|^{p-2} \nabla v - \eta^*) \cdot \nabla \xi \,dx + \intO \eta^* \cdot \nabla \xi  \,dx - \intO h \xi \,dx|
			\\
			&=
			|\intO (|\nabla v|^{p-2} \nabla v - \eta^*) \cdot \nabla \xi \,dx - \intO (\text{div}\,\eta^* + h) \xi \,dx|
			\\
			&\leq
			|\intO (|\nabla v|^{p-2} \nabla v - \eta^*) \cdot \nabla \xi \,dx| + |\intO (\text{div}\,\eta^* + h) \xi \,dx|
			\\
			&\leq
			\||\nabla v|^{p-2} \nabla v - \eta^*\|_{p'} \|\nabla \xi\|_p
			+
			\|\text{div}\,\eta^* + h\|_{p'}
			\|\xi\|_p
			\\
			\label{eq:thm1:proof:1}
			&\leq
			\|\nabla \xi\|_p \left(\||\nabla v|^{p-2} \nabla v - \eta^*\|_{p'} 
			+
			C_F \|\text{div}\,\eta^* + h\|_{p'}\right).
		\end{align}
		Combining \eqref{eq:thm1:proof:2x} with \eqref{eq:thm1:proof:1}, we derive
		\begin{equation}\label{eq:thm1:proof:2}
		2^{2-p} \|\nabla u - \nabla v\|_p^{p-1}
		\leq
		\|F'(v)\|_{*}
		\leq
		\||\nabla v|^{p-2} \nabla v - \eta^*\|_{p'}
		+
		C_F \|\text{div}\,\eta^* + h\|_{p'},
		\end{equation}
		which implies the estimate \eqref{eq:thm:1}.
	\end{proof}

	A similar result is true in the case $p < 2$.
	\begin{theorem}\label{thm:2}
		Let $p < 2$. 
		Then for the solution $u$ of \eqref{eq:weak1} and any $v \in \W$ the following estimate holds:
		\begin{equation}\label{eq:thm:2}
			\|\nabla (u-v)\|_p
			\leq
			C \, \|F'(v)\|_*
			\leq
			C \,
			\inf_{\eta^* \in Q^*}
			\left(
			\|\tau^* - \eta^*\|_{p'} + C_F \|\mathrm{div}\, \eta^* + h\|_{p'}
			\right),
		\end{equation}
		where $\tau^* := |\nabla v|^{p-2} \nabla v$ and $C = (p-1)^{-1} \, 2^{\frac{2-p}{p}} \left(C_F^{p'}\|h\|_{p'}^{p'} + \|\nabla v\|_p^p \right)^{\frac{2-p}{p}}$.
	\end{theorem}
	\begin{proof}
		The proof is similar to that of Theorem~\ref{thm:1} and is a special case of the general approach described in Section~\ref{sec:method}.
		Let us give the necessary details.
		
		The upper bound on $-\langle F'(v), u-v \rangle$ repeats \eqref{eq:thm1:proof:1x} and \eqref{eq:thm1:proof:1}.
		The lower bound is a little more technical and uses the following algebraic inequality instead of \eqref{eq:algebr-ineq1} (see the estimate on \cite[Chapter~12, p.~100]{Lind} combined with $\int_0^1 |a+t(b-a)|^{p-2}\,dt \geq (|a|+|b|)^{p-2}$ for $p<2$):
		\begin{equation}\label{eq:algebr-ineq2}
		(|b|^{p-2}b - |a|^{p-2}a) \cdot (b-a) 
		\geq 
		(p-1) \frac{|b-a|^2}{(|b|+|a|)^{2-p}}
		\quad a,b \in \mathbb{R}^N, \quad p \in (1,2),
		\end{equation}
		where we assume that the right-hand side to be zero when $a=b=0$.
		Consecutively using the H\"older inequality, \eqref{eq:algebr-ineq2}, the inequality $(1+|t|)^p \leq 2 (1+|t|^p)$, and the energy estimate \eqref{eq:energy-id1}, we get
		\begin{align}
		 &\intO |\nabla (u-v)|^p \,dx
		 =
		 \intO \left(\frac{|\nabla (u-v)|^{2}}{(|\nabla u| + |\nabla v|)^{2-p}}\right)^{\frac{p}{2}} (|\nabla u| + |\nabla v|)^{\frac{p(2-p)}{2}} \,dx\\
		 &\leq
		 \left(\intO \frac{|\nabla (u-v)|^2}{(|\nabla u| + |\nabla v|)^{2-p}} \,dx \right)^{\frac{p}{2}}
		 \left(\intO (|\nabla u| + |\nabla v|)^{p} \, dx\right)^{\frac{2-p}{2}}\\
		 &\leq
		 \left(\frac{1}{p-1} \intO (|\nabla u|^{p-2} \nabla u - |\nabla v|^{p-2} \nabla v) \cdot (\nabla u- \nabla v) \,dx\right)^{\frac{p}{2}}
		 \left( 2\intO |\nabla u|^p \,dx + 2 \intO |\nabla v|^{p} \, dx\right)^{\frac{2-p}{2}}\\
		 \label{eq:thm2:p<2}
		 &\leq
		 \left(\frac{1}{p-1} \intO (|\nabla u|^{p-2} \nabla u - |\nabla v|^{p-2} \nabla v) \cdot (\nabla u- \nabla v) \,dx\right)^{\frac{p}{2}}
		 2^{\frac{2-p}{2}}\left(C_F^{p'}\|h\|_{p'}^{p'} + \|\nabla v\|_p^p\right)^{\frac{2-p}{2}}.
		\end{align}
		From here and from \eqref{eq:Fu-v}, we obtain the lower bound
		\begin{align}
			-\langle F'(v), u-v \rangle
			&=
			\intO (|\nabla u|^{p-2} \nabla u - |\nabla v|^{p-2} \nabla v) \cdot (\nabla u - \nabla v) \,dx
			\\
			\label{eq:thm2:proof:0}
			&\geq
			\frac{(p-1) \, 2^{\frac{p-2}{p}}}{\left(C_F^{p'}\|h\|_{p'}^{p'} + \|\nabla v\|_p^p \right)^{\frac{2-p}{p}}} \, \|\nabla (u - v)\|_p^2.
		\end{align}
	Combining this estimate with \eqref{eq:thm1:proof:1x} and \eqref{eq:thm1:proof:1}, we arrive at \eqref{eq:thm:2}.
	\end{proof}
	
	\begin{remark}
		Theorems~\ref{thm:1} and \ref{thm:2} have the following ``geometric'' interpretation: if the derivative $F'(v)$ is small, then $v$ is close to a critical point.
		And since the critical point of $F$ is unique, $v$ is close to $u$.
		The proof shows that the smallness of $F'(v)$ in the direction $u-v$ is sufficient for the analysis.
		
		According to the general approach (see \eqref{eq:abstract-main00}), a natural measure of approximation of $v$ to $u$ estimated through $\|F'(v)\|_{*}$ is the quantity
		$$
		\frac{-\langle F'(v), u-v \rangle}{\|\nabla (u-v)\|_p} 
		=
		\frac{\intO (|\nabla u|^{p-2} \nabla u - |\nabla v|^{p-2} \nabla v) \cdot (\nabla u - \nabla v) \,dx}{\intO |\nabla (u-v)|^{p} \,dx},
		$$
		and Theorems~\ref{thm:1} and \ref{thm:2} give convenient lower bounds for it in terms of the norm.
	\end{remark}
	
		\begin{remark}\label{rem:regQ-discussion}
		Results similar to Theorems~\ref{thm:1} and \ref{thm:2} were obtained in \cite{P1}
		(see also related estimates in \cite{BR,Rep2000-1,RepBook1} and estimates of other types in \cite{BL,CK})
		using estimates of dual functionals, which made their derivation quite bulky.
		Our proof of Theorems~\ref{thm:1}, \ref{thm:2} is elementary and transparent.
		For example, in the case $p \geq 2$, the estimate \eqref{eq:thm:1} without the intermediate term $\|F'(v)\|_*$ is deduced in \cite{P1} from the following abstract estimate obtained in \cite[Theorem~3.1]{P1}:
		\begin{equation}\label{eq:past}
		\frac{1}{p} 2^{1-p} \|\nabla (u-v)\|_p^p 
		\leq
		\inf_{q^* \in Q^*_h} \left(F(v) + \int_\Omega \frac{|q^*|^{p'}}{p'} \,dx\right),
		\end{equation}
		where
		$$
		Q^*_h := 
		\{
		q^* \in L^{p'}(\Omega;\mathbb{R}^N):~ \text{div}\, q^* + h = 0
		\}.
		$$
		Evaluating the majorant in \eqref{eq:past} is equivalent to solving a dual problem to \eqref{eq:weak1}, which is no less difficult than solving the original problem.
		Roughening the right-hand side in \eqref{eq:past}, we write the majorant of deviation of $v$ from $u$ with a fixed function $q^* \in Q^*_h$:
		\begin{align}
		M(v,q^*)
		:=
		F(v) + \int_\Omega \frac{|q^*|^{p'}}{p'} \,dx 
		&=
		\int_\Omega \frac{|\nabla v|^{p}}{p} \,dx
		+
		\int_\Omega \frac{|q^*|^{p'}}{p'} \,dx
		-
		\intO h v \,dx\\
		&=
		\int_\Omega \left(
		\frac{|\nabla v|^{p}}{p} 
		+
		\frac{|q^*|^{p'}}{p'}
		-
		q^* \cdot \nabla v
		\right)dx \geq 0.
		\end{align}
		The non-negativity of $M(v,q^*)$ follows from Young's inequality, and equality is observed only if the elements $q^*$ and $\nabla v$ are conjugate to each other, i.e., $q^* = |\nabla v|^{p-2} \nabla v$.
		Then, since $q^* \in Q^*_h$, we have
		$$
		0 
		= 
		h + \text{div}\, q^*
		=
		h + \text{div}(|\nabla v|^{p-2} \nabla v),
		$$
		that is, $v=u$ is a solution of \eqref{eq:weak1}.
		
		The estimate \eqref{eq:past} appeared in \cite{P1} as a rewritten version of the following inequality obtained via the duality theory: 
		\begin{equation}\label{eq:past2}
		\frac{1}{p} 2^{1-p} \|\nabla (u-v)\|_p^p 
		\leq 
		F(v) - F(u).
		\end{equation}
		The right-hand side of \eqref{eq:past2} can be estimated differently than shown above.
		Namely, due to the convexity of the functional $F$, the following inequality takes place:
		$$
		F(v) - F(u) \leq -\langle F'(v), u-v\rangle.
		$$
		Therefore, using the norm $\|F'(v)\|_*$, we get
		\begin{equation}\label{eq:past3}
		\frac{1}{p} 2^{1-p} \|\nabla (u-v)\|_p^p  \leq \|F'(v)\|_* \|\nabla (u-v)\|_p,
		\end{equation}
		which is analogous to the first estimate in \eqref{eq:thm:1}, but with a rougher constant.
		From here on, the proof of \cite[Theorem~3.3]{P1} can be simplified by taking into account the derivation of the estimate \eqref{eq:thm1:proof:1}. 
		
		This discussion leads to an interesting question about the comparison of the majorants $M(v,q^*)$ and $\|F'(v)\|_*$, which is not addressed in the present work.
	\end{remark}
	
		\begin{remark}\label{rem:assumptions-Poisson}
		Note that $\|F'(v)\|_* \to 0$ if $v \to u$ strongly in $\W$.
		This follows from the equality $F'(u) = 0$ in $W^{-1,p'}(\Omega)$ and the application of Lemma~\ref{lem:flows1} to the estimate
			\begin{align}
				&\|F'(v)\|_* 
				=
				\|F'(u) - F'(v)\|_* 
				\\
				&= 
				\sup_{\xi \in \W \setminus \{0\}} \frac{|\intO (|\nabla u|^{p-2} \nabla u - |\nabla v|^{p-2} \nabla v) \cdot \nabla \xi \,dx|}{\|\nabla \xi\|_p}
				\leq
				\||\nabla u|^{p-2} \nabla u - |\nabla v|^{p-2} \nabla v\|_{p'}.
			\end{align}
		Thus, the estimates via $\|F'(v)\|_*$ in Theorems~\ref{thm:1} and \ref{thm:2} are consistent a posteriori estimates.
		Moreover, it is easy to see that the estimates via $\|F'(v)\|_*$ remain valid if we only assume $h \in W^{-1,p'}(\Omega)$ and use the more general energy estimate $\|\nabla u\|_p^p \leq \|h\|_*^{p'}$ instead of \eqref{eq:energy-id1}.
	\end{remark}
	
	It has already been noted that the quantity $\|\nabla (u-v)\|_p$ controls the deviation of the corresponding fluxes.
	Namely, combining Lemma~\ref{lem:flows1} (taking into account the energy estimate \eqref{eq:energy-id1} in the case $p > 2$) and Theorems~\ref{thm:1}, \ref{thm:2},
	we obtain an a posteriori control of the norm $\||\nabla u|^{p-2} \nabla u - |\nabla v|^{p-2} \nabla v\|_{p'}$ for any function $v \in \W$ with explicit constants.
	
	\medskip
	The quantity $\|\nabla (u-v)\|_p$ can also be estimated from below in the following constructive way, using the convexity of the functional $F$.
	\begin{proposition}\label{prop:lower-bound1}
		Let $p>1$. 
		Then for the solution $u$ of \eqref{eq:weak1} and any $v \in \W$ the following assertions hold:
		\begin{enumerate}[label={\rm(\roman*)}]
			\item\label{lem:lower-bound1:1}
			If $p \leq 2$, then
		\begin{equation}\label{eq:prop:lower-bound1}
			\|\nabla (u-v)\|_p^p
			\geq
			2^{p-2} \sup_{w \in \W}(F(v) - F(w)) \geq 0.
		\end{equation}
		\item\label{lem:lower-bound1:2}
		If $p \geq 2$, then
		\begin{equation}\label{eq:prop:lower-bound2}
			\|\nabla (u-v)\|_p^2
			\geq
			(p-1)^{-1} \, 2^{\frac{2-p}{p}}
			\left(C_F^{p'} \|h\|_{p'}^{p'} + \|\nabla v\|_p^p\right)^{\frac{2-p}{p}} \sup_{w \in \W}(F(v) - F(w)) \geq 0.
		\end{equation}
		\end{enumerate}
	\end{proposition}
	\begin{proof}
		For any $p>1$, using the convexity of $s \mapsto |s|^p$, we get
		$$
		|b|^p \geq |a|^p + p |a|^{p-2} a \cdot (b-a), \quad a,b \in \mathbb{R}^N,
		$$
		and consequently, for any $v \in \W$, we have 
		\begin{equation}\label{eq:convexity-2x}
			-\intO |\nabla v|^{p-2} \nabla v \cdot \nabla (u-v) \,dx 
			\geq
			\frac{1}{p} \|\nabla v\|_p^p 
			-
			\frac{1}{p} \|\nabla u\|_p^p.
		\end{equation}
		Therefore, we obtain
		\begin{align}
		\intO (|\nabla u|^{p-2} \nabla u &- |\nabla v|^{p-2} \nabla v) \cdot \nabla (u-v) \,dx
		=
		\intO h(u-v) \,dx 
		- 
		\intO |\nabla v|^{p-2} \nabla v \cdot \nabla (u-v) \,dx\\
		\label{eq:main-ineq1}
		&\geq
		\intO h(u-v) \,dx 
		+
		\frac{1}{p} \|\nabla v\|_p^p 
		-
		\frac{1}{p} \|\nabla u\|_p^p
		=
		F(v) - F(u) \geq F(v) - F(w),
		\end{align}
		where the last inequality holds for any $w \in \W$, since $u$ is a global minimizer of $F$.	
		
		Now we will use the same inequalities as in the proof of Lemma~\ref{lem:flows1}.
		In the case $p \leq 2$, \eqref{eq:algebr4-0} implies
		$$
		(|b|^{p-2} b - |a|^{p-2}a) \cdot (b-a) 
		\leq
		||b|^{p-2} b - |a|^{p-2}a| \, |b-a|
		\leq 2^{2-p} |b-a|^p, \quad a, b \in \mathbb{R}^N.
		$$
		Combining with \eqref{eq:main-ineq1}, we deduce that		
		\begin{align}\label{eq:energy-equiv1}
		2^{2-p} \|\nabla (u-v)\|_p^p 
		&\geq 
		\intO (|\nabla u|^{p-2} \nabla u - |\nabla v|^{p-2} \nabla v) \cdot \nabla (u-v) \,dx \geq F(v) - F(w),
		\end{align}
		which yields \eqref{eq:prop:lower-bound1}.
		
	 In the case $p \geq 2$, \eqref{eq:algebr4} gives (cf.~\eqref{eq:algebr-ineq2})
		\begin{equation}\label{eq:algebr4x}
			(|b|^{p-2} b - |a|^{p-2}a) \cdot (b-a) 
			\leq
			||b|^{p-2} b - |a|^{p-2} a| \, |b-a| \leq (p-1) (|a|+|b|)^{p-2} |b-a|^2, \quad a,b \in \mathbb{R}^N.
		\end{equation}
	Consequently using \eqref{eq:algebr4x}, the H\"older inequality, and $(1+|t|)^p \leq 2 (1+|t|^p)$, we deduce the estimate
		\begin{align}
			&\intO (|\nabla u|^{p-2} \nabla u - |\nabla v|^{p-2} \nabla v) \cdot \nabla (u-v) \,dx\\
			&\leq
			(p-1) \intO (|\nabla u| + |\nabla v|)^{p-2} |\nabla (u-v)|^2 \,dx\\
			&\leq
			(p-1)
			\left(\intO (|\nabla u| + |\nabla v|)^{p} \,dx\right)^{\frac{p-2}{p}}
			\left( \intO |\nabla (u-v)|^{p} \,dx\right)^{\frac{2}{p}}\\
			\label{eq:lem:flows1:2x}
			&\leq 
			(p-1) 2^{\frac{p-2}{p}}
			\left(\intO |\nabla u|^p \,dx + \intO |\nabla v|^{p}  \,dx\right)^{\frac{p-2}{p}}
			\left( \intO |\nabla (u-v)|^{p} \,dx\right)^{\frac{2}{p}}.
		\end{align}
	Combining \eqref{eq:lem:flows1:2x}, \eqref{eq:energy-id1}, and \eqref{eq:main-ineq1}, we get
	 $$
	 (p-1) 2^{\frac{p-2}{p}}
	 \left(C_F^{p'} \|h\|_{p'}^{p'} + \|\nabla v\|_p^p\right)^{\frac{p-2}{p}}
	 \|\nabla (u-v)\|_p^2
	 \geq 
	 F(v) - F(w),
	 $$
	 which implies \eqref{eq:prop:lower-bound2}.
	\end{proof}
	
	In the case $p=2$, a fact similar to Proposition~\ref{prop:lower-bound1} is proved in \cite[Theorem~3.6]{RepBook1}.
	It is easy to see that the supremum on the right-hand side
	of the estimates \eqref{eq:prop:lower-bound1}, \eqref{eq:prop:lower-bound2} is attained for $w=u$ and vanishes for $v=w$.
	
	\begin{remark}
		The results of this section can be easily generalized, e.g., to the problem
		\begin{equation}\label{eq:pois:gen}
			\left\{
			\begin{aligned}
				-\Delta_p u - a \Delta_q u + |u|^{r-2} u &= h \quad \text{in}~ \Omega,\\
				u &= 0 \quad \text{on}~ \partial\Omega,
			\end{aligned}
			\right.
		\end{equation}
		for any $a \geq 0$, $1<q<p$, $r \in (1, p^*)$, where $p^* := pN/(N-p)$ for $p < N$ and $p^* := +\infty$ for $p \geq N$. 
		The functional space is $\W$, and the monotonicity of the left-hand side of the problem \eqref{eq:pois:gen} is subject to the monotonicity of the $p$-Laplacian.
	\end{remark}

	\section{Obstacle problem}\label{sec:obstacle}	
	
	Again, as in \eqref{eq:F}, we consider the functional
	$$
	F(v) = \frac{1}{p}\intO |\nabla v|^p \,dx 
	- 
	\intO h v \,dx, 
	$$
	but minimize it not on the whole $\W$, but on a closed convex set
	$$
	K = 
	\{v \in \W:~ v \geq \phi ~\text{a.e.\ in}~ \Omega\},
	$$
	where $\phi \in L^1(\Omega)$ is such that $K \neq \emptyset$. 
	The function $\phi$ is usually called an obstacle.
	It is known that the functional $F$ on $K$ has a unique minimizer $u \in K$ and it satisfies the variational (integral) inequality
	\begin{equation}\label{eq:F'>0}
	\langle F'(u), w-u \rangle 
	= 
	\intO |\nabla u|^{p-2} \nabla u \cdot \nabla (w-u) \,dx - \intO h(w-u) \,dx \geq 0, \quad w \in K.
	\end{equation}
	If $\langle F'(u), u \rangle = 0$, then $u$ satisfies the energy estimate \eqref{eq:energy-id1}. 
	For instance, if $K$ is a cone with apex at zero, then the choice $w=0$ and $w=2u$ in \eqref{eq:F'>0} yields $\langle F'(u), u \rangle = 0$. 
	This equality is also valid under the assumption $\phi \leq 0$ and $h \geq 0$ in $\Omega$,
	since the solution is non-negative by the maximum principle, and therefore the choice $w=0$ and $w=2u$ in \eqref{eq:F'>0} is possible.
	We present the following result about estimates on $u$, which is valid in the general case.
	\begin{lemma}\label{lem:energy-est}
		Let $p>1$. Then for any $w \in K$ the solution $u \in K$ of \eqref{eq:F'>0} satisfies the estimates
		\begin{align}
			\label{eq:upper-bound-on-u2}
			&\|\nabla u\|_p 
			\leq 
			C_1 :=
			\max\{(p C_F \|h\|_{p'}+1)^{\frac{1}{p-1}}, 
			p F(w)\},\\
			\label{eq:upper-bound-on-u2x}
			&\|\nabla u\|_p 
			\leq 
			C_2 :=
			\frac{1}{(1-\varepsilon (1+C_F^p))^{1/p}}
			\left(
			\frac{1}{\varepsilon^{p-1} } \|\nabla w\|_p^p
			+
			\frac{1}{\varepsilon^{\frac{1}{p-1}}} \|h\|_{p'}^{p'}
			+
			\|h\|_{p'} \|w\|_p
			\right)^{1/p},
		\end{align}
		where $\varepsilon \in (0,(1+C_F^p)^{-1})$ is arbitrary.
	\end{lemma}
	\begin{proof}
		Let us prove \eqref{eq:upper-bound-on-u2}.
		Since $u, w \in K$ and $u$ is a minimizer of $F$ on $K$, we have $F(u) \leq F(w)$, which implies that
		\begin{equation}\label{eq:upper-bound-on-u}
			\|\nabla u\|_p^p 
			\leq 
			p \intO h u \,dx
			+
			p F(w) 
			\leq
			p C_F \|h\|_{p'} \, \|\nabla u\|_p
			+
			p F(w).
		\end{equation}
		Analyzing the function $t \mapsto t^p - A t - B$, where $t, A, B \geq 0$, it is not hard to see that it is positive for any $t > \max\{(A+1)^{\frac{1}{p-1}},B\}$. 
		Therefore, \eqref{eq:upper-bound-on-u} implies \eqref{eq:upper-bound-on-u2}. 
		
		Let us prove \eqref{eq:upper-bound-on-u2x}.
		Using the H\"older and Friedrichs inequalities, as well as Young's inequalities
		$$
		ab = (p' \varepsilon)^{1/p'}a \, \frac{b}{(p' \varepsilon)^{1/p'}} \leq \varepsilon |a|^{p'} + \frac{|b|^{p}}{p (p' \varepsilon)^{p-1}},
		\quad
		ab = \frac{a}{(p \varepsilon)^{1/p}} \, (p \varepsilon)^{1/p} b \leq \frac{|a|^{p'}}{p' (p \varepsilon)^{\frac{1}{p-1}}} + \varepsilon |b|^{p}, 
		$$
		where $a,b \in \mathbb{R}$ and $\varepsilon>0$, 
		we deduce from \eqref{eq:F'>0} the following chain of inequalities:
		\begin{align*}
			\intO |\nabla u|^{p} \,dx 
			&\leq 
			\intO
			|\nabla u|^{p-2} \nabla u \cdot \nabla w \,dx - \intO h(w-u) \,dx
			\\
			&\leq
			\|\nabla u\|_p^{p-1} \|\nabla w\|_p
			+
			\|h\|_{p'} (\|u\|_p + \|w\|_p)
			\\
			&\leq
			\varepsilon \|\nabla u\|_p^p
			+
			\frac{1}{p (p' \varepsilon)^{p-1} } \|\nabla w\|_p^p
			+
			\frac{1}{p' (p \varepsilon)^{\frac{1}{p-1}} } \|h\|_{p'}^{p'}
			+
			\varepsilon C_F^p \|\nabla u\|_p^p
			+
			\|h\|_{p'} \|w\|_p.
		\end{align*}
	Assuming that $\varepsilon \in (0,(1+C_F^p)^{-1})$, we get
	\begin{equation}\label{eq:est1}
		\|\nabla u\|_p^p
		\leq
		\frac{1}{1-\varepsilon (1+C_F^p)}
		\left(
		\frac{1}{p (p' \varepsilon)^{p-1} } \|\nabla w\|_p^p
		+
		\frac{1}{p' (p \varepsilon)^{\frac{1}{p-1}} } \|h\|_{p'}^{p'}
		+
		\|h\|_{p'} \|w\|_p
		\right).
	\end{equation}
	Roughening the denominators on the right-hand side of \eqref{eq:est1} for simplicity, we derive \eqref{eq:upper-bound-on-u2x}.
	\end{proof}

	The advantage of the estimate \eqref{eq:upper-bound-on-u2x} is that its derivation does not use the functional $F$. Therefore, it can be extended to cases where the operator of a problem is not potential (see, e.g., Section~\ref{sec:px-Poisson}).
	Note also that if $\phi_+ \in W_0^{1,p}(\Omega)$, then one can choose $w=\phi_+$ in the estimates of Lemma~\ref{lem:energy-est}.
	
	\medskip
	For any $v \in K$, we define the quantity (see~\eqref{eq:K*})
	\begin{equation}\label{eq:F'K}
	\|F'(v)\|_{*,K} 
	=
	\sup_{w \in K \setminus \{v\}} \frac{(-\langle F'(v), w - v \rangle)_+}{\|\nabla (w - v)\|_p}
	\end{equation}
	and denote
	$$
	\Omega^v_\phi = \{x \in \Omega:~ v(x)=\phi(x)\}, 
	\quad 
	\Omega_0 = \Omega \setminus \Omega^v_\phi.
	$$
	
	\begin{theorem}\label{thm:obstacle}
		Let $p>1$.
		Then for the solution $u$ of \eqref{eq:F'>0} and any $v \in K$ the following estimates hold:	
		\begin{align}
			\label{eq:obstacle-main:p>2}
			2^{2-p} \|\nabla (u-v)\|_p^{p-1}
			+
			\frac{\langle F'(u), v-u \rangle}{\|\nabla (u-v)\|_p}
			&\leq
			\|F'(v)\|_{*,K},   \quad \text{if}~ p \geq 2,\\
			\label{eq:obstacle-main:p<2}
			\frac{(p-1) \, 2^{\frac{p-2}{p}}}{\left(C_i^p + \|\nabla v\|_p^p \right)^{\frac{2-p}{p}}} \|\nabla (u-v)\|_p
				+
			\frac{\langle F'(u), v-u \rangle}{\|\nabla (u-v)\|_p}
			&\leq
			\|F'(v)\|_{*,K},   \quad \text{if}~ p < 2,
		\end{align}
		and
		\begin{equation}\label{eq:obstacle-main:F}
			\|F'(v)\|_{*,K}
			\leq
			\inf_{\eta^* \in Q^*}
			\left(\|\tau^* - \eta^*\|_{p'} 
			+
			C_F \|\mathrm{div}\,\eta^* + h \|_{p',\Omega^v_0}
			+
			C_F \|(\mathrm{div}\,\eta^* + h)_+\|_{p',\Omega^v_\phi}
			\right),
		\end{equation}
	where $\tau^* := |\nabla v|^{p-2} \nabla v$ and $C_i > 0$ is any of the two constants defined in \eqref{eq:upper-bound-on-u2} and \eqref{eq:upper-bound-on-u2x} for a fixed $w \in W_0^{1,p}(\Omega)$. 
	The second term on the left-hand sides of \eqref{eq:obstacle-main:p>2} and \eqref{eq:obstacle-main:p<2} is assumed to be zero when $v=u$.
	
	In particular, since $\langle F'(u), v-u \rangle \geq 0$, we have
	\begin{equation}\label{eq:obstacle-main}
		\|\nabla (u-v)\|_p
		\leq
		C \|F'(v)\|_{*,K}^{\frac{1}{\max\{1,p-1\}}}, 
	\end{equation}	
	where 
	\begin{equation}
	C = 
		\begin{cases}
			2^{\frac{p-2}{p-1}},   &\text{if}~ p \geq 2,\\
			(p-1)^{-1} \, 2^{\frac{2-p}{p}} \left(C_i^p + \|\nabla v\|_p^p \right)^{\frac{2-p}{p}},   &\text{if}~ p < 2.
		\end{cases}
	\end{equation} 
	\end{theorem}
	\begin{proof}
	For any $v \in K$, we estimate from below and from above the expression
	\begin{align}
	-\langle F'(v), u-v \rangle
	&=
	-\intO |\nabla v|^{p-2} \nabla v \cdot \nabla (u-v) \,dx + \intO h (u-v) \,dx
	\\
	&=
	\intO (|\nabla u|^{p-2} \nabla u - |\nabla v|^{p-2} \nabla v) \cdot (\nabla u - \nabla v) \,dx
	+
	\langle F'(u), v-u \rangle.
	\end{align}
	In the case $p \geq 2$, similarly to the derivation of the estimate \eqref{eq:thm1:proof:0} we have
	\begin{equation}\label{eq:thm3:proof:0}
	-\langle F'(v), u-v \rangle
	\geq
	2^{2-p} \|\nabla (u-v)\|_p^{p} +
	\langle F'(u), v-u \rangle.
	\end{equation}
	In the case $p<2$, the same reasoning as in the derivation of the estimate \eqref{eq:thm2:proof:0} also applies, except that instead of the energy estimate $\|\nabla u\|_p^p \leq C_F^{p'}\|h\|_{p'}^{p'}$ in \eqref{eq:thm2:p<2} (which might not be true for the obstacle problem), any bound $\|\nabla u\|_p^p \leq C_i^p$ from Lemma~\ref{lem:energy-est} is used.
	Thus, arguing as the proof of \eqref{eq:thm2:proof:0} and taking into account the additional term $\langle F'(u), v-u \rangle$, we obtain 
	\begin{equation}\label{eq:thm3:proof:0x}
		-\langle F'(v), u-v \rangle
		\geq
		\frac{(p-1) \, 2^{\frac{p-2}{p}}}{\left(C_i^p + \|\nabla v\|_p^p \right)^{\frac{2-p}{p}}} \|\nabla (u-v)\|_p^{2}
		+
		\langle F'(u), v-u \rangle.
	\end{equation}
	On the other hand, using the notation \eqref{eq:F'K}, we have the upper bound
	\begin{equation}\label{eq:thm1obst:proof:1x}
		-\langle F'(v), u-v \rangle \leq \|F'(v)\|_{*,K} \, \|\nabla (u-v)\|_p.
	\end{equation}
	Combining \eqref{eq:thm3:proof:0} (for $p \geq 2$), \eqref{eq:thm3:proof:0x} (for $p < 2$), and \eqref{eq:thm1obst:proof:1x}, we obtain \eqref{eq:obstacle-main:p>2} and \eqref{eq:obstacle-main:p<2}.	
	In order to get \eqref{eq:obstacle-main:F}, we will argue similarly to the derivation of the estimate \eqref{eq:thm1:proof:1}. 
	For any $p>1$ and arbitrary $w \in K$ and $\eta^* \in Q^*$, we have 
	\begin{align}
		&-\langle F'(v), w-v \rangle\\
		&=
		-\intO (|\nabla v|^{p-2} \nabla v - \eta^*) \cdot \nabla (w-v) \,dx 
		- \intO \eta^* \cdot \nabla (w-v)  \,dx 
		+ \intO h (w-v) \,dx
		\\
		&=
		-\intO (|\nabla v|^{p-2} \nabla v - \eta^*) \cdot \nabla (w-v) \,dx 
		+ 
		\intO (\text{div}\,\eta^* + h) (w-v) \,dx
		\\
		\notag
		&\leq
		-\intO (|\nabla v|^{p-2} \nabla v - \eta^*) \cdot \nabla (w-v) \,dx 
		+ \int_{\Omega^v_0} (\text{div}\,\eta^* + h) (w-v) \,dx 
		+ 
		\int_{\Omega^v_\phi} (\text{div}\,\eta^* + h)_+ (w-v) \,dx
		\\
		\label{eq:thm3:proof:1}
		&\leq
		\|\nabla (w-v)\|_p \left(\||\nabla v|^{p-2} \nabla v - \eta^*\|_{p'} 
		+
		C_F \|\text{div}\,\eta^* + h\|_{p',\Omega^v_0}
		+
		C_F \|(\text{div}\,\eta^* + h)_+\|_{p',\Omega^v_\phi}
		\right),
	\end{align}
	which yields \eqref{eq:obstacle-main:F}.
	Here, to derive the first inequality, we used the fact that $w \geq \phi = v$ in $\Omega^v_\phi$, and therefore
	$$
	\int_{\Omega^v_\phi} (\text{div}\,\eta^* + h) (w-v) \,dx
	\leq
	\int_{\Omega^v_\phi} (\text{div}\,\eta^* + h)_+ (w-v) \,dx.
	\qedhere
	$$
	\end{proof}
	
	\begin{remark}
		Suppose that $|\nabla u|^{p-2} \nabla u \in Q^*$ and 
		\begin{equation}\label{eq:obstacle-0}
			-\Delta_p u \geq h,
			\quad
			(\Delta_p u + h) (u-\phi) = 0
			\quad \text{a.e.\ in}~ \Omega.
		\end{equation}
		Then the right-hand side of the estimate \eqref{eq:obstacle-main:F} vanishes for $v=u$ and $\eta^* = |\nabla u|^{p-2} \nabla u$. 
		Some sufficient conditions on the data of the problem guaranteeing \eqref{eq:obstacle-0} can be found, in particular, in \cite{caselli,fuchs,rodrig} and \cite[Chapter~2]{lions}, see also Remark~\ref{rem:regQ}.
	\end{remark}

	\begin{remark}\label{rem:conv}
		In general, it is not guaranteed that the majorant in \eqref{eq:obstacle-main:F} tends to zero as $v \to u$ strongly in $\W$.
		It depends on the data of the problem and the choice of the sequence $\{v\}$.
		For example, suppose that \eqref{eq:obstacle-0} holds, $|\nabla u|^{p-2} \nabla u \in Q^*$, and $\eta^* = |\nabla u|^{p-2} \nabla u$. 
		Then, assuming $\Omega_\phi^v \supset \Omega_\phi^u$, we have
		$$
		\|\text{div}\,\eta^* + h\|_{p',\Omega^v_0}
		+
		\|(\text{div}\,\eta^* + h)_+\|_{p',\Omega^v_\phi} = 0,
		$$
		and $\||\nabla v|^{p-2} \nabla v - \eta^*\|_{p'} \to 0$ as $v \to u$ in $\W$ by Lemma~\ref{lem:flows1}. 
		Thus, the convergence to zero on the right-hand side of \eqref{eq:obstacle-main:F} takes place.
		On the other hand, if $\Omega_\phi^v \subset \Omega_\phi^u$ and $\Omega_0^v \cap \Omega_\phi^u \neq \emptyset$, then 
		$$
		\|\mathrm{div}\,\eta^* + h \|_{p',\Omega^v_0}
		+
		\|(\mathrm{div}\,\eta^* + h)_+\|_{p',\Omega^v_\phi}
		=
		\|\mathrm{div}\,\eta^* + h \|_{p',\Omega^v_0},
		$$
		and this expression may not tend to zero as $v \to u$ in $\W$, in general.
	\end{remark}

	\begin{remark}
		Combining Theorem~\ref{thm:obstacle} and Lemma~\ref{lem:flows1}, we obtain an a posteriori control of the deviation of the flux of any $v \in K$ from the flux of the solution $u$ of the problem \eqref{eq:F'>0} with explicitly calculated constants.
	\end{remark}

	\begin{remark}\label{rem:ANR}
	A result similar to Theorem~\ref{thm:obstacle} was obtained in \cite{ANR} using the duality theory methods.
	In particular, in the case $p > 2$, under stronger assumptions on the regularity of the obstacle $\phi$, with the choice $\eta^* = |\nabla v|^{p-2} \nabla v$, and for sufficiently small $\epsilon>0$, \cite[Theorem~2]{ANR} gives 
		\begin{equation}\label{eq:ANR}
		\left(c(p) - \epsilon^p \frac{C_F^p}{p}\right) \|\nabla (u-v)\|_p^p
		+
		\int_{\Omega^u_\phi} (\Delta_p \phi + h) (u-v) \,dx 
		+ 
		\lambda^*(\eta^*)
		\leq
		\frac{1}{p' \epsilon^{p'}} \|R_h(\mathbf{\eta^*})\|_{p'}^{p'}.
	\end{equation}
	Here, $c(p)>0$ is an explicitly estimated constant from a certain algebraic inequality characterizing the convexity of a power functional,
	\begin{align}
	\label{eq:lambda*}
	\lambda^*(\eta^*)
	&=
	\intO \left(\frac{1}{p} |\nabla u|^{p} + \frac{1}{p'} |\eta^*|^{p'} - \nabla u \cdot \eta^*\right) dx,\\
	R_h(\mathbf{\eta^*})
	&=
	\left\{
	\begin{aligned}
		&\text{div}\,\eta^* + h 	&&\text{on}~ \Omega_0^v,\\
		&(\text{div}\,\eta^* + h)_+ &&\text{on}~ \Omega_0^\phi.
	\end{aligned}
	\right.
	\end{align}
	At the same time, under the assumption \eqref{eq:obstacle-0} and with the same choice of $\eta^* = |\nabla v|^{p-2} \nabla v$, the estimates \eqref{eq:obstacle-main:p>2} and \eqref{eq:obstacle-main} from Theorem~\ref{thm:obstacle} take the form
	\begin{gather}\label{eq:BP10}
		2^{2-p} \|\nabla (u-v)\|_p^{p-1}
		+
		\frac{\int_{\Omega^u_\phi} (\Delta_p \phi + h)  (u-v) \,dx}{\|\nabla (u-v)\|_p}
		\leq 
		C_F
		\|R_h(\eta^*)\|_{p'},\\
		\label{eq:BP1x}
		\|\nabla (u-v)\|_p^p
		\leq
		2^{p' (p-2)}
		C_F^{p'}
		\|R_h(\eta^*)\|_{p'}^{p'}.
	\end{gather}
	The estimate \eqref{eq:BP10} can be multiplied by $\|\nabla (u-v)\|_p$ and then Young's inequality can be used on the right-hand side.
	The resulting estimate will differ from \eqref{eq:ANR} only by the term $\lambda^*(\eta^*)$, which determines a measure of deviation of the solution's flux from $\eta^* = |\nabla v|^{p-2} \nabla v$.
	However, by Lemma~\ref{lem:flows1}, the norm of deviation of the flux of $v$ from the flux of $u$ is estimated through the norm $\|\nabla (u-v)\|_p$.
	Note that the estimates of Theorem~\ref{thm:obstacle} are derived much simpler than \eqref{eq:ANR}, they do not require additional regularity of the obstacle, and uniformly cover any $p>1$.
	
	We also note that, due to the presence of normalization, the term $\frac{\int_{\Omega^u_\phi} (\Delta_p \phi + h) (u-v) \,dx}{\|\nabla (u-v)\|_p}$ in \eqref{eq:BP10} better reflects the relation between the coincidence sets $\Omega^u_\phi$ and $\Omega^v_\phi$ than the term $\int_{\Omega^u_\phi} (\Delta_p \phi + h) (u-v) \,dx$ in \eqref{eq:ANR}.
	In particular, the entire left-hand side of \eqref{eq:ANR} tends to zero for \textit{any} sequence $v \to u$ in $\W$, $\{v\} \subset K \setminus \{u\}$, while the quantity $\|R_h(\eta^*)\|_{p'}$ on the right-hand side of \eqref{eq:ANR} may not tend to zero, see Remark~\ref{rem:conv}.
	\end{remark}

	\section{Other variants of the Poisson problem}\label{sec:collection}
	\subsection{Anisotropic Poisson problem}\label{sec:px-Poisson}
	
	Let us consider the problem with an anisotropic $p$-Laplacian
	\begin{equation}\label{eq:anisotropic}
		\left\{
		\begin{aligned}
			-\text{div}\, (|\nabla u|^{p-2} \mathcal{A} \nabla u) &= h \quad \text{in}~ \Omega,\\
			u &= 0 \quad \text{on}~ \partial\Omega.
		\end{aligned}
		\right.
	\end{equation}
	Here, the matrix $\mathcal{A} = \mathcal{A}(x)$ is assumed to be measurable, symmetric, and such that
	\begin{equation}\label{eq:anisotropic:A}
	\nu |\xi|^2 
	\leq
	\lambda(x) |\xi|^2
	\leq 
	\mathcal{A}(x) \xi \cdot \xi 
	\leq
	\Lambda(x) |\xi|^2 
	\leq
	\nu^{-1} |\xi|^2, \quad \xi \in \mathbb{R}^N,~ x \in \Omega,
	\end{equation}
	for some $\nu>0$, where $\lambda = \lambda(x)$ and $\Lambda = \Lambda(x)$ are minimal and maximal eigenvalues of $\mathcal{A}$, respectively. 
	Let, in addition to \eqref{eq:anisotropic:A}, the exponent $p$ and the matrix $\mathcal{A}$ satisfy the following condition (a Cordes type condition):
	\begin{equation}\label{eq:anisotropin:assumption}
	\delta_\Omega(p,\mathcal{A}) 
	:= 
	\inf_{x \in \Omega} 
	\left(
	p - |p-2| \, \mu(\mathcal{A})
	\right) > 0,
	\end{equation}
	where
	$$
	\mu(\mathcal{A}) = \sup_{\xi \neq 0} \frac{|\mathcal{A} \xi| \, |\xi|}{\mathcal{A}\xi \cdot \xi}.
	$$
	In \cite[Theorem~1.4]{SZ}, it is shown that the assumption \eqref{eq:anisotropin:assumption} guarantees 
	$$
	(|b|^{p-2}\mathcal{A}(x)b - |a|^{p-2}\mathcal{A}(x)a) \cdot (b-a) 
	> 0,
	\quad a,b \in \mathbb{R}^N, ~a \neq b, ~ x \in \Omega,
	$$
	which implies the strict monotonicity of the operator,
	see also \cite[Lemma~3.2]{PY} and Appendix~\ref{sec:appendix}.
	Note that, in the general case, without additional assumptions on the symmetric positive-definite matrix $\mathcal{A}$, the monotonicity of the anisotropic $p$-Laplacian may be violated, see a constructive example in \cite[(1.17)]{SZ}.
	
	Strict monotonicity of the operator guarantees that the problem \eqref{eq:anisotropic} has the unique weak solution (see \cite[Theorem~3.1]{PY}).
	Recall that a weak solution to \eqref{eq:anisotropic} is a function $u \in \W$ satisfying 
	\begin{equation}\label{eq:anisotropic:weak}
		\intO |\nabla u|^{p-2} \mathcal{A}(x) \nabla u \cdot \nabla \xi \,dx
		=
		\intO h \xi \,dx 
		\quad \text{for any}~ \xi \in \W.
	\end{equation}
	The solution $u$ satisfies the energy estimate
	\begin{equation}\label{eq:energy-anis}
	\|\nabla u\|_p^p 
	\leq 
	\nu^{-p'} C_F^{p'} \|h\|_{p'}^{p'},
	\end{equation}
	which follows from
	$$
	\nu \|\nabla u\|_p^p \leq 
	\intO |\nabla u|^{p-2} \mathcal{A}(x) \nabla u \cdot \nabla u \,dx 
	= 
	\intO h u\,dx \leq \|h\|_{p'} \|u\|_{p}
	\leq
	C_F \|h\|_{p'} \|\nabla u\|_p.
	$$
	
	Let us denote the operator of the problem as $A: \W \to W^{-1,p'}(\Omega)$, so that \eqref{eq:anisotropic:weak} can be written as $\langle A u - h, \xi \rangle = 0$ for any $\xi \in \W$. 

	\begin{theorem}
		Let $p>1$ and the assumptions \eqref{eq:anisotropic:A} and \eqref{eq:anisotropin:assumption} be satisfied. 
		Then for the solution $u$ of \eqref{eq:anisotropic} and any $v \in \W$ the following estimate holds:			
		\begin{equation}\label{eq:thm:anisotropic}
			\|\nabla (u-v)\|_p
			\leq
			C\|Av-h\|_*^{\frac{1}{\max\{1,p-1\}}}
			\leq 
			C \inf_{\eta^* \in Q^*} \left(\|\tau^* - \eta^*\|_{p'} 
			+
			C_F \|\mathrm{div}\,\eta^* + h \|_{p'}
			\right)^{\frac{1}{\max\{1,p-1\}}},
		\end{equation}
		where $\tau^* := |\nabla v|^{p-2} \mathcal{A} \nabla v$ and 
		\begin{equation}
			C = 
			\begin{cases}
				\left(\frac{8 (p-1)}{\inf_{x \in \Omega}\lambda(x) \, \delta_\Omega(p,\mathcal{A})}\right)^{\frac{1}{p-1}},   &\text{if}~ p \geq 2,\\
				\frac{2 \left(\nu^{-p'} C_F^{p'}\|h\|_{p'}^{p'} + \|\nabla v\|_p^p \right)^{\frac{2-p}{p}}}{\inf_{x \in \Omega}\lambda(x) \, \delta_\Omega(p,\mathcal{A}) \, 2^{\frac{p-2}{p}}},   &\text{if}~ p < 2.
			\end{cases}
		\end{equation}
	\end{theorem}
	\begin{proof}
		Let us estimate from below and from above the expression 
		\begin{align}
		-\langle A v - h, u-v \rangle 
		&= 
		-\intO |\nabla v|^{p-2} \mathcal{A}(x) \nabla v \cdot \nabla (u-v) \,dx
		+
		\intO h (u-v) \,dx\\
		&=
		\intO (|\nabla u|^{p-2} \mathcal{A}(x) \nabla u - |\nabla v|^{p-2} \mathcal{A}(x) \nabla v) \cdot \nabla (u-v) \,dx.
		\end{align}
		In the case $p \geq 2$, using \eqref{eq:algebr-ineq1:anisotrop:p>2:1}, we get
		\begin{equation}\label{eq:anisotropic:proof:1}
			-\langle A v - h, u-v \rangle 
			\geq 
			\frac{\inf_{x \in \Omega}\lambda(x) \, \delta_\Omega(p,\mathcal{A})}{8 (p-1)} \, \|\nabla (u-v)\|_p^p.
		\end{equation}
		In the case $p < 2$, applying \eqref{eq:algebr-ineq1:anisotrop:p<2:1} and arguing as in the derivation of \eqref{eq:thm2:proof:0} with the energy estimate \eqref{eq:energy-anis}, we deduce
		\begin{equation}\label{eq:anisotropic:proof:1x}
		-\langle A v - h, u-v \rangle
		\geq
		\frac{\inf_{x \in \Omega}\lambda(x) \, \delta_\Omega(p,\mathcal{A}) \, 2^{\frac{p-2}{p}}}{2 \left(\nu^{-p'} C_F^{p'}\|h\|_{p'}^{p'} + \|\nabla v\|_p^p \right)^{\frac{2-p}{p}}} \, \|\nabla (u - v)\|_p^2.
		\end{equation}		
		On the other hand, we have the upper bound
		\begin{equation}\label{eq:thm1anisot:proof:1x}
			-\langle A v - h, u-v \rangle 
			\leq 
			\|A v - h\|_{*} \, \|\nabla (u-v)\|_p.
		\end{equation}
		Combining \eqref{eq:anisotropic:proof:1}, \eqref{eq:anisotropic:proof:1x}, and \eqref{eq:thm1anisot:proof:1x}, we get the first estimate in \eqref{eq:thm:anisotropic}:
		\begin{equation}\label{eq:thm1anicotr:proof:2x}
			C_0 \|\nabla (u-v)\|_p^{\max\{1,p-1\}}
			\leq
			\|Av-h\|_*,
		\end{equation}
		where 
		\begin{equation}
			C_0 = 
			\begin{cases}
				\frac{\inf_{x \in \Omega}\lambda(x) \, \delta_\Omega(p,\mathcal{A})}{8 (p-1)},   &\text{if}~ p \geq 2,\\
				\frac{\inf_{x \in \Omega}\lambda(x) \, \delta_\Omega(p,\mathcal{A}) \, 2^{\frac{p-2}{p}}}{2 \left(\nu^{-p'} C_F^{p'}\|h\|_{p'}^{p'} + \|\nabla v\|_p^p \right)^{\frac{2-p}{p}}},   &\text{if}~ p < 2.
			\end{cases}
		\end{equation}	
		In order to estimate $\|Av-h\|_*$, we take arbitrary $\xi \in \W$ and $\eta^* \in Q^*$ and, as in \eqref{eq:thm1:proof:1}, obtain
		\begin{align}
			|\langle A v - h, \xi \rangle|
			&\leq
			|\intO (|\nabla v|^{p-2} \mathcal{A}(x) \nabla v - \eta^*) \cdot \nabla \xi \,dx
			+
			\intO \eta^* \cdot \nabla \xi \,dx
			-
			\intO h \xi \,dx|\\
			&=
			|\intO (|\nabla v|^{p-2} \mathcal{A}(x) \nabla v - \eta^*) \cdot \nabla \xi \,dx
			-
			\intO (\text{div}\,\eta^* +h) \xi \,dx|\\
			&\leq
			|\intO (|\nabla v|^{p-2} \mathcal{A}(x) \nabla v - \eta^*) \cdot \nabla \xi \,dx|
			+
			|\intO (\text{div}\,\eta^* +h) \xi \,dx|\\
			&\leq
			\label{eq:anisotropic:proof:2}
			\|\nabla \xi\|_{p}
			\left(\||\nabla v|^{p-2} \mathcal{A} \nabla v - \eta^*\|_{p'} 
			+
			C_F \|\text{div}\,\eta^* +h\|_{p'} 
			\right).
		\end{align}
		Combining \eqref{eq:thm1anicotr:proof:2x} with \eqref{eq:anisotropic:proof:2}, we derive 
		$$
		\|Av-h\|_*
		\leq
		\||\nabla v|^{p-2} \mathcal{A} \nabla v - \eta^*\|_{p'} 
		+
		C_F \|\text{div}\,\eta^* +h\|_{p'},
		$$
		which yields the second part of the estimate \eqref{eq:thm:anisotropic}. 
	\end{proof}
	
	\begin{remark}
		Arguing as in Remark~\ref{rem:assumptions-Poisson}, it is not hard to see that if $v \to u$ strongly in $\W$, then $\|Av-h\|_* \to 0$ by Lemma~\ref{lem:flows-anisotrop}.
		Moreover, Lemma~\ref{lem:flows-anisotrop} implies that $\|\nabla (w-v)\|_{p}$ controls the difference of the corresponding anisotropic fluxes, as in the case of the identity matrix $\mathcal{A}$ considered in Section~\ref{sec:p-Poisson}.
	\end{remark}

\subsection{Poisson problem with Neumann boundary conditions}\label{sec:Neumann}

In Section~\ref{sec:p-Poisson}, the Dirichlet problem for the Poisson equation was studied.
Now we consider the Neumann problem for the Poisson equation: find $u \in W^{1,p}(\Omega)$ satisfying the normalization condition $\intO u \,dx = 0$ and the integral identity
\begin{equation}\label{eq:neumann:weak1}
	\intO |\nabla u|^{p-2} \nabla u \cdot \nabla \xi \,dx
	=
	\intO h \xi \,dx
	\quad \text{for any}~ \xi \in W^{1,p}(\Omega). 
\end{equation}
Here $\intO h \,dx = 0$, and $\Omega$ is bounded and sufficiently regular, for example, Lipschitz.
Such a function $u$ exists and is unique.

The following Poincar\'e inequality is known:
\begin{equation}\label{eq:poincare1}
	\|w\|_p \leq C_P \|\nabla w\|_p
	\quad \text{for any}~ w \in W^{1,p}(\Omega) ~\text{such that}~ \intO w \,dx = 0.
\end{equation}
Upper bounds on the constant $C_P$ for star-shaped $\Omega$ can be found, e.g., in \cite{FR}.
We also note the trace inequality 
\begin{equation}\label{eq:poincare}
\|w\|_{p,\partial \Omega} 
\leq C_T \|\nabla w\|_p
\quad \text{for any}~ w \in W^{1,p}(\Omega) ~\text{such that}~ \int_{\Omega} w \,dx = 0,
\end{equation}
see, e.g., \cite[Theorem~1.5.1.10]{grisvard}, from which one can also derive some estimates on $C_T$. 

For any $w \in W^{1,p}(\Omega)$, we denote by $t_w \in \mathbb{R}$ such a constant that 
$\intO (w+t_w)\,dx = 0$. 
Consider the norm
$$
\|F'(v)\|_{*,W^{1,p}} := \sup_{w \in W^{1,p}(\Omega) \setminus \{0\}, ~\intO w \,dx = 0} \frac{|\langle F'(v),w\rangle|}{\|\nabla w\|_p}
$$
and the space
\begin{equation}\label{eq:Q*1}
	Q^*_\nu = \{\eta \in L^{p'}(\Omega;\mathbb{R}^N):~ \text{div}\, \eta \in L^{p'}(\Omega), ~ \eta \cdot \nu \in L^{p'}(\partial\Omega)\},
\end{equation}
where $\nu$ is a unit outward normal vector to $\partial\Omega$. 

\begin{theorem}\label{thm:neumann}
	Let $p>1$.
	Then for the solution $u$ of \eqref{eq:neumann:weak1} and any $v \in W^{1,p}(\Omega)$ the following estimate holds:		
	\begin{align}
		&\|\nabla (u-v)\|_p
		\leq
		C \|F'(v)\|_{*,W^{1,p}}^{\frac{1}{\max\{1,p-1\}}}
		\\
		\label{eq:neumann-main}
		&\leq
		C
		\inf_{\eta^* \in Q^*_\nu}
		\left(
		\|\tau^* - \eta^*\|_{p'} 
		+
		C_P \|\mathrm{div}\,\eta^* + h\|_{p'}
		+
		C_T \|\eta^* \cdot \nu\|_{p',\partial \Omega}
		\right)^{\frac{1}{\max\{1,p-1\}}},
	\end{align}
	where $\tau^* := |\nabla v|^{p-2} \nabla v$ and 
	\begin{equation}
		C = 
		\begin{cases}
			2^{\frac{p-2}{p-1}},   &\text{if}~ p \geq 2,\\
			(p-1)^{-1} \, 2^{\frac{2-p}{p}} \left(C_P^{p'}\|h\|_{p'}^{p'} + \|\nabla v\|_p^p \right)^{\frac{2-p}{p}},   &\text{if}~ p < 2.
		\end{cases}
	\end{equation} 
\end{theorem}
\begin{proof}
Let us estimate the expression
\begin{align}
-\langle F'(v),u-v \rangle
&=
-\intO |\nabla v|^{p-2} \nabla v \cdot \nabla (u-v) \,dx
+
\intO h (u-v) \,dx,\\
\label{eq:neumann:proof:0x}
&=
\intO (|\nabla u|^{p-2} \nabla u - |\nabla v|^{p-2} \nabla v) \cdot \nabla (u-v) \,dx.
\end{align}
The right-hand side of \eqref{eq:neumann:proof:0x} is estimated as in \eqref{eq:thm1:proof:0} for $p \geq 2$ and \eqref{eq:thm2:proof:0} for $p \leq 2$, with the only difference being that in \eqref{eq:thm2:proof:0} the Friedrichs constant $C_F$ is replaced by the constant $C_P$ from the Poincar\'e inequality \eqref{eq:poincare1}.
Since $\intO h \,dx = 0$, we write the upper bound as
\begin{equation}\label{eq:neumann:proof:1x}
	-\langle F'(v), u-v \rangle
	=
	-\langle F'(v), (u-v)+t_{u-v} \rangle \leq \|F'(v)\|_{*,W^{1,p}} \|\nabla (u-v)\|_p.
\end{equation}
The upper estimate on $\|F'(v)\|_{*,W^{1,p}}$ is similar to \eqref{eq:thm1:proof:1} and can be derived as follows. 
For any $\xi \in W^{1,p}(\Omega)$ such that $\intO \xi \,dx = 0$, and for any $\eta^* \in Q^*_\nu$, we have 
\begin{align}
	|\langle F'(v), \xi \rangle|
	&=
	|\intO (|\nabla v|^{p-2} \nabla v - \eta^*) \cdot \nabla \xi \,dx + \intO \eta^* \cdot \nabla \xi  \,dx - \intO h \xi \,dx|
	\\
	&=
	|\intO (|\nabla v|^{p-2} \nabla v - \eta^*) \cdot \nabla \xi \,dx - \intO (\text{div}\,\eta^* + h) \xi \,dx + \int_{\partial \Omega} \xi \, (\eta^* \cdot \nu) \,dS|
	\\
	&\leq
	|\intO (|\nabla v|^{p-2} \nabla v - \eta^*) \cdot \nabla \xi \,dx| + |\intO (\text{div}\,\eta^* + h) \xi \,dx|
	+
	|\int_{\partial \Omega} \xi \, (\eta^* \cdot \nu) \,dS|
	\\
	\label{eq:neumann:proof:1}
	&\leq
	\|\nabla \xi\|_p \left(\||\nabla v|^{p-2} \nabla v - \eta^*\|_{p'} 
	+
	C_P \|\text{div}\,\eta^* + h\|_{p'}
	+
	C_T \|\eta^* \cdot \nu\|_{p',\partial \Omega} 
	\right),
\end{align}
where the inequalities \eqref{eq:poincare1} and \eqref{eq:poincare} have been used.
Thus, we obtain the estimate \eqref{eq:neumann-main}.
\end{proof}

\begin{remark}
In the linear case $p=2$, related estimates are presented in \cite[Chapter~4.1.3]{RepBook1}.
Some results on local and global regularity of the solution of the problem \eqref{eq:neumann:weak1} are obtained in \cite{antonini,ciacni-mazja,guarnotta,lieberman}, see also Remark~\ref{rem:regQ}. 
\end{remark}

\subsection{Vector Poisson problem}\label{sec:vector}

Consider the vector version of the Poisson problem from Section~\ref{sec:p-Poisson}.
Let $n \in \mathbb{N}$. 
Using classical methods, it can be proved that there exists a unique vector-function
$\mathbf{u} = (u_1,\dots,u_n) \in W_0^{1,p}(\Omega;\mathbb{R}^n)$ such that 
\begin{equation}\label{eq:vector1}
\intO |\nabla \mathbf{u}|^{p-2} \nabla \mathbf{u} \cdot \nabla \boldsymbol{\xi} \,dx = \intO \mathbf{h} \cdot \boldsymbol{\xi} \,dx, \quad \boldsymbol{\xi} \in W_0^{1,p}(\Omega;\mathbb{R}^n),
\end{equation}
where 
\begin{align}
|\nabla \mathbf{u}| 
&= 
(|\nabla u_1|^2 + \dots + |\nabla u_n|^2)^{1/2},\\
\nabla \mathbf{u} \cdot \nabla \boldsymbol{\xi} 
&=
\nabla u_1 \cdot \nabla \xi_1 + \dots + \nabla u_n \cdot \nabla \xi_n,\\
\mathbf{h} \cdot \boldsymbol{\xi}
&=
h_1 \xi_1 + \dots + h_n \xi_n.
\end{align}
Here, we assume that $\mathbf{h} \in L^{p'}(\Omega;\mathbb{R}^n)$. 

Introduce the space 
\begin{equation}\label{eq:Q*-vec}
	Q^*_{\text{vec}} = \{\boldsymbol{\eta} = (\eta_1,\dots,\eta_n) \in (L^{p'}(\Omega;\mathbb{R}^{N}))^n:~ \textbf{div}\, \boldsymbol{\eta} \in L^{p'}(\Omega; \mathbb{R}^n)\},
\end{equation}
where the divergence is defined componentwise:
$$
\textbf{div}\, \boldsymbol{\eta} = (\text{div}\, \eta_1, \dots, \text{div}\, \eta_n).
$$

We have the following result, which is similar to Theorems~\ref{thm:1} and \ref{thm:2}. 
In its proof, the algebraic inequalities \eqref{eq:algebr-ineq1} and \eqref{eq:algebr-ineq2} are used for vectors of dimension $nN$.
\begin{theorem}\label{thm:vector-Poisson}
	Let $p>1$ and $n \in \mathbb{N}$.
	Then for the solution $\mathbf{u}$ of  \eqref{eq:vector1} and any $\mathbf{v} \in W_0^{1,p}(\Omega;\mathbb{R}^n)$ the following estimate holds:		
	\begin{equation}\label{eq:vector-Poisson-main}
		\|\nabla (\mathbf{u}-\mathbf{v})\|_p
		\leq
		C \|F'(\mathbf{v})\|_{*}^{\frac{1}{\max\{1,p-1\}}}
		\leq
		C
		\inf_{\boldsymbol{\eta^*} \in Q^*_{\mathrm{vec}}}
		\left(\|\boldsymbol{\tau^*} - \boldsymbol{\eta^*}\|_{p'} 
		+
		C_F \|\mathbf{div}\,\boldsymbol{\eta^*} + \mathbf{h} \|_{p'}
		\right)^{\frac{1}{\max\{1,p-1\}}},
	\end{equation}
	where $\boldsymbol{\tau^*} := |\nabla \mathbf{v}|^{p-2} \nabla \mathbf{v}$ and 
	\begin{equation}
		C = 
		\begin{cases}
			2^{\frac{p-2}{p-1}},   &\text{if}~ p \geq 2,\\
			(p-1)^{-1} \, 2^{\frac{2-p}{p}} \left(C_F^{p'}\|\mathbf{h}\|_{p'}^{p'} + \|\mathbf{v}\|_p^p \right)^{\frac{2-p}{p}},   &\text{if}~ p < 2.
		\end{cases}
	\end{equation} 
\end{theorem}

\section{Polyharmonic obstacle problem}\label{sec:polyharm}

Let $\Omega$ be a bounded domain in $\mathbb{R}^N$, $N \geq 1$, and let $m \in \mathbb{N}$, $m \geq 2$.  
For sufficiently smooth functions, we define a nonlinear polyharmonic operator as
\begin{equation}\label{eq:nonlinear-polyharm}
	\Delta_p^m u = 
	\begin{cases}
		\Delta^{k}(|\Delta^{k}u|^{p-2}\Delta^{k}u),   &\text{if}~ m=2k,\\
		-\Delta^{k}(\text{div}\,(|\nabla (\Delta^{k}u)|^{p-2}\nabla (\Delta^{k}u)),  &\text{if}~ m=2k+1, 
	\end{cases}
\quad k \in \mathbb{N},
\end{equation}
where $\Delta = \text{div} \nabla$ is the Laplacian. 
In the case $m=1$, $\Delta_p^m$ is the $p$-Laplacian. 
If $p=2$ and $m=2$, then $\Delta_p^m$ is a bi-Laplacian. 
We also introduce the following operator of order $m$:
\begin{equation}\label{eq:nonlinear-polyharm:symbol}
	\mathcal{D}_m u 
	= 
	\begin{cases}
		\Delta^{k}u,   &\text{if}~ m=2k,\\
		\nabla (\Delta^{k}u),  &\text{if}~ m=2k+1,
	\end{cases}
	\quad k \in \mathbb{N}.
\end{equation}
Consider the space $W_0^{m,p}(\Omega)$ as the closure of $C_0^\infty(\Omega)$ with respect to the norm of $W^{m,p}(\Omega)$.
It is known (see \cite{CP}) that $W_0^{m,p}(\Omega)$ can be equipped with the norm $\|\mathcal{D}_m (\cdot)\|_p$ and this norm is equivalent to the standard one.
Moreover, there is a compact embedding of $W_0^{m,p}(\Omega)$ into $L^p(\Omega)$.

As in Section~\ref{sec:obstacle}, consider
$$
K = \{v \in W_0^{m,p}(\Omega):~ v \geq \phi~ \text{a.e.\ in}~ \Omega\},
$$
where the obstacle $\phi \in L^1(\Omega)$ is such that $K \neq \emptyset$. 
The set $K$ is closed and convex in $W_0^{m,p}(\Omega)$. 
By general theory, there exists a unique function $u \in K$ satisfying the inequality
\begin{equation}\label{eq:polyharm:u}
\langle \Delta_p^m u - h, w-u \rangle \geq 0, 
\quad w \in K.
\end{equation}
In integral terms, $u$ satisfies 
$$
\int_\Omega |\mathcal{D}_m u|^{p-2} \mathcal{D}_m u \, \mathcal{D}_m (w-u) \,dx \geq \int_\Omega h (w-u) \,dx,
$$
or, in the expanded form,
\begin{align}
	\int_\Omega |\Delta^k u|^{p-2} \Delta^k u \, \Delta^k (w-u) \,dx &\geq 
	\int_\Omega h (w-u) \,dx,
	\quad \text{if}~ m=2k,\\
	\int_\Omega |\nabla(\Delta^k u)|^{p-2} \,\nabla(\Delta^k u) \cdot \nabla(\Delta^k (w-u)) \,dx
	&\geq 
	\int_\Omega h (w-u) \,dx,
	\quad \text{if}~ m=2k+1.
\end{align}
In particular, if $K = W_0^{m,p}(\Omega)$, then $u$ is a weak solution of the Dirichlet problem for the polyharmonic Poisson equation
\begin{equation}\label{eq:Poisson:polyharm}
	\left\{
	\begin{aligned}
		\Delta_p^m u &= h \quad \text{in}~ \Omega,\\
		u = 0, \dots, \frac{\partial^{m-1} u}{\partial \nu^{m-1}} &= 0 \quad \text{on}~ \partial\Omega.
	\end{aligned}
	\right.
\end{equation}
Similarly to Lemma~\ref{lem:energy-est} (see the inequality \eqref{eq:upper-bound-on-u2}), 
it can be shown that, for an arbitrary $w \in K$, the solution $u \in K$ of \eqref{eq:polyharm:u} satisfies the estimate 
\begin{equation}\label{eq:polyharm:energyest}
\|\mathcal{D}_m u\|_p \leq 
C_3 
:=
\max\{(p C_{F,m} \|h\|_{p'}+1)^{\frac{1}{p-1}}, 
\|\mathcal{D}_m w\|_p - p \,\langle h, w\rangle\},
\end{equation}
where $C_{F,m}$ is the constant of embedding  $W_0^{m,p}(\Omega) \subset L^p(\Omega)$. 
For simplicity, we do not provide an analogue of the inequality \eqref{eq:upper-bound-on-u2x}.

As in Section~\ref{sec:p-Poisson}, $C_{F,m}$ is expressed in terms of the first eigenvalue $\lambda_{1,m}(p;\Omega)$ of the polyharmonic $p$-Laplacian: $C_{F,m} = \lambda_{1,m}^{-1/p}(p;\Omega)$.
Upper and lower bounds on $\lambda_{1,m}(p;\Omega)$ in terms of $\lambda_{1,m}(p;B_R)$ with a suitable choice of $R>0$ follow from the domain monotonicity of the first eigenvalue, which is a consequence of the variational formulation of $\lambda_{1,m}(p;B_R)$ and the definition of $W_0^{m,p}(\Omega)$. 
Constructive estimates on $C_{F,m}$ in the case $p=2$ can be found, for example, in \cite{YY}.

For an arbitrary function $v \in {K}$, we denote
$$
\Omega^v_\phi = \{x \in \Omega:~ v(x)=\phi(x)\}, 
\quad 
\Omega^v_0 = \Omega \setminus \Omega^v_\phi.
$$
As in  \eqref{eq:K*}, consider the ``norm over $K$''
$$
\|\Delta_p^m v - h\|_{*,{K}}
=
\sup_{w \in K \setminus \{v\}} \frac{
	(-\langle \Delta_p^m v - h, w-v \rangle)_+}{\|w - v\|_X},
\quad v \in K,
$$
and the space
$$
Q^*_m 
= 
\begin{cases}
\{\eta \in L^{p'}(\Omega):~ \Delta^k \eta \in L^{p'}(\Omega)\},  &\text{if}~ m=2k, \\
\{\eta \in L^{p'}(\Omega;\mathbb{R}^N):~ \Delta^k (\text{div}\, \eta) \in L^{p'}(\Omega)\},  &\text{if}~ m=2k+1.
\end{cases}
$$

Using the general approach described in Section~\ref{sec:method}, we prove the following a posteriori estimates.
\begin{theorem}\label{thm:polyharm}
	Let $p>1$, $m \in \mathbb{N}$, and $m \geq 2$. 
	Then for the solution $u \in K$ of \eqref{eq:polyharm:u} and any $v \in K$ the following estimates hold:		
	\begin{align}
		\label{eq:polyharm-main0:p>2}
		2^{2-p} \|\mathcal{D}_m(u-v)\|_p^{p-1}
		+
		\frac{\langle \Delta_p^m u - h, v-u \rangle}{\|\mathcal{D}_m(u-v)\|_p}
		&\leq
		\|\Delta_p^m v - h\|_{*,{K}},   \quad \text{if}~ p \geq 2,\\
		\label{eq:polyharm-main0:p<2}
		\frac{(p-1) \, 2^{\frac{p-2}{p}}}{\left(C_3^p + \|\nabla v\|_p^p \right)^{\frac{2-p}{p}}} \|\mathcal{D}_m(u-v)\|_p
		+
		\frac{\langle \Delta_p^m u - h, v-u \rangle}{\|\mathcal{D}_m(u-v)\|_p}
		&\leq
		\|\Delta_p^m v - h\|_{*,{K}},   \quad \text{if}~ p < 2,
	\end{align}
	where the constant $C_{3}>0$ is given in \eqref{eq:polyharm:energyest} for an arbitrary fixed $w \in K$. 
	Moreover, if $m=2k$, then 
	\begin{align}
		&\|\Delta_p^m v - h\|_{*,{K}}
		\\
		\label{eq:polyharm-main1}
		&\leq
		\inf_{\eta^* \in Q^*_m}
		\left(\|\tau^*  - \eta^*\|_{p'} 
		+
		C_{F,m}\|\Delta^k \eta^*  - h\|_{p',\Omega_0^v} 
		+
		C_{F,m}\|(\Delta^k \eta^*  - h)_+\|_{p',\Omega_\phi^v}
		\right),
	\end{align}
	where $\tau^* := |\Delta^k v|^{p-2}\Delta^k v$, 
	while if $m=2k+1$, then 
	\begin{align}
		&\|\Delta_p^m v - h\|_{*,{K}}
		\\
		\label{eq:polyharm-main2}
		&\leq
		\inf_{\eta^* \in Q^*_m}
		\left(
		\|\tau^*  - \eta^*\|_{p'}
		+
		C_{F,m}\|\Delta^k (\mathrm{div}\, \eta^*) + h\|_{p',\Omega_0^v}
		+
		C_{F,m}\|(\Delta^k (\mathrm{div}\, \eta^*) + h)_+\|_{p',\Omega_\phi^v}
		\right),
	\end{align}
	where $\tau^* := |\nabla (\Delta^k v)|^{p-2}\nabla (\Delta^k v)$.
	
	In particular, since $\langle \Delta_p^m u - h, v-u \rangle \geq 0$, we have
	\begin{equation}\label{eq:polyharm-main0}
		\|\mathcal{D}_m(u-v)\|_p
		\leq
		C \|\Delta_p^m v - h\|_{*,{K}}^{\frac{1}{\max\{1,p-1\}}},
	\end{equation}
	where 
	\begin{equation}
		C = 
		\begin{cases}
			2^{\frac{p-2}{p-1}},   &\text{if}~ p \geq 2,\\
			(p-1)^{-1} \, 2^{\frac{2-p}{p}} \left(C_3^p + \|\nabla v\|_p^p \right)^{\frac{2-p}{p}},   &\text{if}~ p < 2.
		\end{cases}
	\end{equation} 
\end{theorem}
\begin{proof}
	For any $v \in K$, we estimate the expression 
	\begin{align}
	-&\langle \Delta_p^m v - h, u-v \rangle
	=
	\langle \Delta_p^m u -  \Delta_p^m v, u-v \rangle
	+
	\langle \Delta_p^m u - h, v-u \rangle\\
	&=
	\intO (|\mathcal{D}_m u|^{p-2} \mathcal{D}_m u - |\mathcal{D}_m v|^{p-2} \mathcal{D}_m v)(\mathcal{D}_m u - \mathcal{D}_m v) \,dx
	+
	\langle \Delta_p^m u - h, v-u \rangle.
	\end{align}
	In the case $p \geq 2$,  the lower bound follows from \eqref{eq:algebr-ineq1}:
	\begin{equation}\label{eq:polyharm:p<21}
	-\langle \Delta_p^m v - h, u-v \rangle \geq 2^{2-p} 
	\|\mathcal{D}_m(u-v)\|_p^p + \langle \Delta_p^m u - h, v-u \rangle.
	\end{equation}
	In the case $p < 2$, the lower bound follows from arguments similar to the derivation of the estimate \eqref{eq:thm2:proof:0}. 
	Namely,  
	\begin{equation}\label{eq:polyharm:p>21}
		-\langle \Delta_p^m v - h, u-v \rangle
		\geq 
		\frac{(p-1) \, 2^{\frac{p-2}{p}}}{\left(C_3^p + \|\mathcal{D}_m v\|_p^p \right)^{\frac{2-p}{p}}} \|\mathcal{D}_m (u-v)\|_p^{2} + \langle \Delta_p^m u - h, v-u \rangle,
	\end{equation}
	where $C_3>0$ is defined in \eqref{eq:polyharm:energyest}. 
	
	The upper bound is written directly using the notation $\|(-\Delta)^m v - h\|_{*,{K}}$:
	\begin{equation}\label{eq:polyharm:p<22}
	-\langle \Delta_p^m v - h, u-v \rangle
	\leq 
	\|\Delta_p^m v - h\|_{*,{K}} \|\mathcal{D}_m (u - v)\|_{p}.
	\end{equation}
	Combining \eqref{eq:polyharm:p<21}--\eqref{eq:polyharm:p<22},
	we get \eqref{eq:polyharm-main0:p>2}, \eqref{eq:polyharm-main0:p<2}, \eqref{eq:polyharm-main0}. 
	\noeqref{eq:polyharm:p>21} \noeqref{eq:polyharm-main0:p<2} 
	The estimates \eqref{eq:polyharm-main1} and \eqref{eq:polyharm-main2} on $\|\Delta_p^m v - h\|_{*,{K}}$ are implied by the following calculations with arbitrary $w \in {K}$ and $\eta^* \in Q^*_m$. 
	For $m=2k$, 
	\begin{align}
	-&\langle \Delta_p^m v - h, w-v \rangle
	\\
	&=
	-\int_\Omega (|\Delta^k v|^{p-2}\Delta^k v  - \eta^*) \, \Delta^k (w-v) \,dx
	-
	\int_\Omega \eta^* \, \Delta^k (w-v) \,dx
	+
	\int_\Omega h (w-v) \,dx
	\\
	&=
	-\int_\Omega (|\Delta^k v|^{p-2}\Delta^k v  - \eta^*) \, \Delta^k (w-v) \,dx
	-
	\int_\Omega (\Delta^k \eta^* - h) (w-v) \,dx
	\\
	&\leq
	\|\tau^*  - \eta^*\|_{p'} \|\Delta^k (w-v)\|_{p}
	+
	\int_{\Omega_0^v} (\Delta^k \eta^* - h) (w-v) \,dx
	+
	\int_{\Omega_\phi^v} (\Delta^k \eta^* - h)_+ (w-v) \,dx
	\\
	&\leq
	\|\Delta^k (w-v)\|_{p}
	\left(
	\|\tau^*  - \eta^*\|_{p'} 
	+
	C_{F,m}\|\Delta^k \eta^*  - h\|_{p',\Omega_0^v} 
	+
	C_{F,m}\|(\Delta^k \eta^*  - h)_+\|_{p',\Omega_\phi^v}
	\right),
	\end{align}
which yields \eqref{eq:polyharm-main1}.
In the case $m=2k+1$, we estimate $\|\Delta_p^m v - h\|_{*,{K}}$ similarly:
	\begin{align}
	&-\langle \Delta_p^m v - h, w-v \rangle
	\\
	&=
	-\int_\Omega (|\nabla(\Delta^k v)|^{p-2} \nabla(\Delta^k v)  - \eta^*) \cdot \nabla(\Delta^k (w-v)) \,dx
	\\
	&\quad-
	\int_\Omega \eta^* \cdot \nabla(\Delta^k (w-v)) \,dx
	+
	\int_\Omega h (w-v) \,dx
	\\
	&=
	-\int_\Omega (|\nabla(\Delta^k v)|^{p-2} \nabla(\Delta^k v)  - \eta^*) \cdot \nabla(\Delta^k (w-v)) \,dx
	+
	\int_\Omega (\Delta^k (\text{div}\, \eta^*) + h) (w-v) \,dx
	\\
	\notag
	&\leq
	\|\nabla(\Delta^k (w-v))\|_{p}
	\left(
	\|\tau^*  - \eta^*\|_{p'}
	+
	C_{F,m}\|\Delta^k (\text{div}\, \eta^*) + h\|_{p',\Omega_0^v}
	+
	C_{F,m}\|(\Delta^k (\text{div}\, \eta^*) + h)_+\|_{p',\Omega_\phi^v}
	\right),
\end{align}
which implies \eqref{eq:polyharm-main2}.
\end{proof}

\begin{remark}
	In the case $p=2$ and $m=2$, 
	the estimates \eqref{eq:polyharm-main0} and \eqref{eq:polyharm-main1} of Theorem \ref{thm:polyharm} give
	\begin{equation}\label{eq:polyharm-main0m2}
			\|\Delta(u-v)\|_2
			\leq
			\|\Delta v  - \eta^*\|_{2} 
			+
			C_{F,2}\|\Delta \eta^*  - h\|_{2,\Omega_0^v} 
			+
			C_{F,2}\|(\Delta \eta^*  - h)_+\|_{2,\Omega_\phi^v}.
	\end{equation}
	This estimate was obtained in \cite{Frolov1,NR,RepBook1} by similar arguments, but without using the ``norm'' $\|\Delta_p^m v - h\|_{*,{K}}$.
	See also \cite{AR,Besov} for an estimate obtained using a dual formulation of the considered problem (cf.\ Remark~\ref{rem:ANR}).
\end{remark}

\begin{remark}
	We refer to \cite[Chapter~2.5]{GGS} and \cite{CF,Fre} for information on the regularity of the solution $u$ in the case $p=2$. 
	In particular, if the obstacle is sufficiently smooth, then
	$$
	(\Delta_2^m u - h)(u - \phi) = 0 
	\quad \text{and} \quad
	\Delta_2^m u - h \geq 0 
	\quad \text{a.e.\ in}~ \Omega.
	$$
	It is then clear that the right-hand sides in \eqref{eq:polyharm-main1} and \eqref{eq:polyharm-main2} vanish for $v = u$ and, respectively, $\eta^* = \Delta^k u$ when $m=2k$ or $\eta^* = \nabla (\Delta^k u)$ when $m=2k+1$.
\end{remark}

\section{Nonlocal Poisson problem}\label{sec:fractional-Poisson}

Let $s \in (0,1)$ and $\Omega$ be an open bounded set in $\mathbb{R}^N$. 
Define the Sobolev space of fractional order
$$
W^{s,p}(\mathbb{R}^N)
= 
\{
u \in L^p(\mathbb{R}^N):~ [u]_p < +\infty
\},
$$
where $[\,\cdot\,]_p$ is the Gagliardo seminorm:
\begin{equation}\label{eq:nonlocal-1}
	[u]_p^p 
	= 
	\iint_{\mathbb{R}^{2N}} \frac{|u(x)-u(y)|^{p}}{|x-y|^{N+ps}} \, dxdy
	=
	\iint_{\mathbb{R}^{2N}} \left(\frac{|u(x)-u(y)|}{|x-y|^{s}}\right)^p \frac{dxdy}{|x-y|^{N}}.
\end{equation}
Let $\Wso$ stand for the closure of $C_0^\infty(\Omega)$ with respect to the norm $\|\cdot\|_p+[\,\cdot\,]_p$ of $W^{s,p}(\mathbb{R}^N)$.
This space is uniformly convex, separable, Banach, and can be equipped with the norm $[\,\cdot\,]_p$.
Moreover, the embedding of $\Wso$ into $L^p(\Omega)$ is compact, see, e.g., \cite{BLP}.
We denote by $C_{F,s}$ the corresponding embedding constant.

Let us discuss known estimates on $C_{F,s}$, see similar estimates on $C_F$ in Section~\ref{sec:p-Poisson}.
Obviously, $C_{F,s} = \lambda_{1,s}^{-1/p}(p;\Omega)$, where $\lambda_{1,s}(p;\Omega)$ is the first eigenvalue of the fractional $p$-Laplacian.
In terms of $\lambda_{1,s}(p;B_R)$, the lower bound on $\lambda_{1,s}(p;\Omega)$ is given by the Faber--Krahn inequality (see \cite[Theorem~3.5]{BLP}), and the upper bound follows, for example, from the domain monotonicity of $\lambda_{1,s}(p;\Omega)$.
Also, $\lambda_{1,s}(p;\Omega)$ can be estimated through some geometric characteristics of the domain (see \cite[Remark~3.2]{BLP}), and in the linear case $p=2$ through the first eigenvalue of the classical Laplacian in $B_R$ (see \cite{chensong}).

Consider the problem of finding a function $u \in \Wso$ such that
\begin{equation}\label{eq:DEu}
	\frac{1}{p} \,\langle D[u]_p^p, \xi \rangle 
	-
	\int_\Omega h\, \xi \, dx
	= 0
	\quad \text{for any}~ \xi \in \Wso,
\end{equation}
where
\begin{equation}\label{eq:Dupxi}
	\langle D[u]_p^p, \xi \rangle 
	=
	p \,
	\iint_{\mathbb{R}^{2N}} \frac{|u(x)-u(y)|^{p-2}(u(x)-u(y))(\xi(x)-\xi(y))}{|x-y|^{N+ps}} \, dxdy.
\end{equation}
This problem can be interpreted as the problem of finding a critical point of the functional $F: \Wso \to \mathbb{R}$, defined as
$$
F(u) = \frac{1}{p} \, [u]_p^p - \int_\Omega h u \, dx.
$$
The solution $u$ of \eqref{eq:DEu} exists and is unique.
Moreover, substituting $\xi = u$ into \eqref{eq:DEu} and applying the H\"older and Friedrichs inequalities, we derive the standard energy estimate
\begin{equation}\label{eq:energy-id2}
	[u]_p^p 
	\leq 
	C_{F,s}^{p'} \|h\|_{p'}^{p'}.
\end{equation}

We define a ``nonlocal gradient'' of a measurable function $w: \mathbb{R}^{N} \mapsto \mathbb{R}$ and a ``nonlocal divergence'' of a measurable function $\gamma: \mathbb{R}^{2N} \mapsto \mathbb{R}$ as follows (see, e.g., \cite{hepp}, and also \cite[Section~8]{BLP} for alternative definitions):
\begin{equation}\label{eq:grad-div-nonlocal}
	\nabla^{(s)} w(x,y) = \frac{w(x)-w(y)}{|x-y|^s},
	\qquad
	(\text{div}^{(s)} \gamma)(x) = 2  
	\int_{\mathbb{R}^N} \gamma(x,y) \, \frac{dy}{|x-y|^{N+s}}.
\end{equation}
We also consider a nonlocal $p$-Laplacian of $w$ (see, e.g., \cite{LL}), which is naturally related to the problem \eqref{eq:DEu}, and a notation related to \eqref{eq:nonlocal-1}:
\begin{equation}\label{eq:nonloc:N}
\mathcal{L}^s_p w(x)
=  
2 \,
\int_{\mathbb{R}^N} \frac{|u(x)-u(y)|^{p-2}(u(x)-u(y))}{|x-y|^{N+ps}} \, dy,
\quad
\mathcal{P}(\gamma) = \iint_{\mathbb{R}^{2N}} |\gamma(x,y)|^{p'} \, \frac{dxdy}{|x-y|^{N}}.
\end{equation}
Evidently, 
\begin{equation}
\mathcal{L}^s_p w 
	=
\text{div}^{(s)}(|\nabla^{(s)} w|^{p-2} \nabla^{(s)} w),
\quad
\mathcal{P}(|\nabla^{(s)} w|^{p-2}\nabla^{(s)} w) = [w]_p^p.
\end{equation}
The presence of the factor $2$ in \eqref{eq:grad-div-nonlocal} and \eqref{eq:nonloc:N} is explained by the expression \eqref{eq:frac:pointwise-weak} below.

Define the space
\begin{equation}\label{eq:Qs}
Q_s^* = \left\{
\gamma: \mathbb{R}^{2N} \mapsto \mathbb{R}:~ \gamma(x,y)=-\gamma(y,x), ~~
\text{div}^{(s)} \gamma \in L^{p'}(\Omega),
~~
\mathcal{P}(\gamma) < \infty
\right\}.
\end{equation}

\begin{theorem}\label{thm:nonlocal}
	Let $p>1$ and $s \in (0,1)$. 
	Then for the solution $u$ of \eqref{eq:DEu} and any $v \in \Wso$ the following estimates hold:	
	\begin{equation}\label{eq:nonlocal-main1}
		[u-v]_p
		\leq
		C \|F'(v)\|_{*}^{\frac{1}{\max\{1,p-1\}}}
	\end{equation}
	and
	\begin{equation}\label{eq:nonlocal-main3}
		\|F'(v)\|_{*} \leq 
		\inf_{\gamma^* \in Q_s^*}
		\left(
		\left(\mathcal{P}(|\nabla^{(s)} v|^{p-2} \nabla^{(s)} v - \gamma^*)\right)^{1/p'} + C_{F,s} \|\mathrm{div}^{(s)} \gamma^* - h\|_{p'}
		\right),
	\end{equation}
	where $\|F'(v)\|_{*}$ is the standard operator norm and  
	\begin{equation}
		C = 
		\begin{cases}
			2^{\frac{p-2}{p-1}},   &\text{if}~ p \geq 2,\\
			(p-1)^{-1} \, 2^{\frac{2-p}{p}} \left(C_{F,s}^{p'} \|h\|_{p'}^{p'} + [v]_p^p \right)^{\frac{2-p}{p}},   &\text{if}~ p < 2.
		\end{cases}
	\end{equation} 
	Moreover, if
	$\mathcal{L}^s_p v \in L^{p'}(\Omega)$,
	then  
	\begin{equation}\label{eq:nonlocal-main2}
		\|F'(v)\|_{*} \leq C_{F,s} \|\mathcal{L}^s_p v - h\|_{p'}.
	\end{equation}
\end{theorem}
\begin{proof}
	Let us estimate the expression $-\langle F'(v),u-v\rangle$ from below:
	\begin{align}
		-\langle F'(v),u-v\rangle
		&=
		-\frac{1}{p} \,\langle D[v]_p^p, u-v \rangle 
		+
		\int_\Omega h\, (u-v) \, dx
		\\
		&=
		\label{eq:nonloc1}
		\frac{1}{p} \,\langle D[u]_p^p, u-v \rangle 
		-\frac{1}{p} \,\langle D[v]_p^p, u-v \rangle 
		=
		\frac{1}{p} \,\langle D[u]_p^p - D[v]_p^p, u-v \rangle. 
	\end{align}
	For the convenience of further analysis, we denote
$$
U(x,y) = u(x)-u(y), 
\quad
V(x,y) = v(x)-v(y).
$$
Then
\begin{equation}\label{eq:nonloc2}
\frac{1}{p} \,\langle D[u]_p^p - D[v]_p^p, u-v \rangle
=
\iint_{\mathbb{R}^{2N}} \frac{(|U|^{p-2}U - |V|^{p-2}V)(U-V)}{|x-y|^{N+ps}} \, dxdy.
\end{equation}

In the case $p \geq 2$, using the inequality \eqref{eq:algebr-ineq1}, we get
\begin{equation}\label{eq:nonloc3}
\iint_{\mathbb{R}^{2N}} \frac{(|U|^{p-2}U - |V|^{p-2}V)(U-V)}{|x-y|^{N+ps}} \, dxdy
\geq
2^{2-p}
\iint_{\mathbb{R}^{2N}} \frac{|U-V|^p}{|x-y|^{N+ps}} \, dxdy
=
2^{2-p} [u-v]_p^p.
\end{equation}

In the case $p < 2$, we argue analogously to the derivation of \eqref{eq:thm2:p<2}. 
Namely, using the H\"older inequality, the estimate \eqref{eq:algebr-ineq2}, and the inequality $(1+|t|)^p \leq 2 (1+|t|^p)$, we obtain
\begin{align}
	&\iint_{\mathbb{R}^{2N}} \frac{|U-V|^p}{|x-y|^{N+ps}} \,dxdy
	=
	\iint_{\mathbb{R}^{2N}} \left(\frac{|U-V|^{2}}{|x-y|^{N+ps}(|U| + |V|)^{2-p}}\right)^{\frac{p}{2}} 
	\left(\frac{(|U| + |V|)^p}{|x-y|^{N+ps}}\right)^{\frac{2-p}{2}} \,dxdy\\
	&\leq
	\left(\iint_{\mathbb{R}^{2N}} \frac{|U-V|^{2}}{|x-y|^{N+ps}(|U| + |V|)^{2-p}} \,dxdy \right)^{\frac{p}{2}}
	\left(\iint_{\mathbb{R}^{2N}} \frac{(|U| + |V|)^p}{|x-y|^{N+ps}}\, dxdy\right)^{\frac{2-p}{2}}\\
	&\leq
	\left(\frac{1}{p-1} \iint_{\mathbb{R}^{2N}} \frac{(|U|^{p-2} U - |V|^{p-2} V) (U- V)}{|x-y|^{N+ps}} \,dxdy\right)^{\frac{p}{2}}
	\left( 2 [u]_p^p + 2 [v]_p^p \right)^{\frac{2-p}{2}}.
\end{align}
From here, due to the energy estimate \eqref{eq:energy-id2}, we have the lower bound
\begin{equation}\label{eq:nonloc4}
	\iint_{\mathbb{R}^{2N}} \frac{(|U|^{p-2}U - |V|^{p-2}V)(U-V)}{|x-y|^{N+ps}} \, dxdy
	\geq
	\frac{(p-1) \, 2^{\frac{p-2}{p}}}{\left(C_{F,s}^{p'} \|h\|_{p'}^{p'} + [v]_p^p\right)^{\frac{2-p}{p}}} \, [u-v]_p^2.
\end{equation}

Let us estimate $-\langle F'(v),u-v\rangle$ from above: 
\begin{equation}\label{eq:nonloc5}
-\langle F'(v),u-v\rangle
\leq 
\|F'(v)\|_{*} \, [u-v]_p.
\end{equation}
Combining this inequality with \eqref{eq:nonloc1}, \eqref{eq:nonloc2}, and either \eqref{eq:nonloc3} for $p \geq 2$ or  \eqref{eq:nonloc4} for $p < 2$, we derive \eqref{eq:nonlocal-main1}. 

Assume now that $\mathcal{L}^s_p v \in L^{p'}(\Omega)$. 
Then, for any function $\xi \in \Wso$, it follows from Fubini's theorem and the skew symmetry arguments that
\begin{align}
	\frac{1}{p}\langle D[v]_p^p, \xi \rangle 
	&=
	\iint_{\mathbb{R}^{2N}} \frac{|v(x)-v(y)|^{p-2}(v(x)-v(y))(\xi(x)-\xi(y))}{|x-y|^{N+ps}} \, dxdy
	\\
	&=2 \int_{\mathbb{R}^N} \xi(x) 
	\int_{\mathbb{R}^N} 
	\frac{|v(x)-v(y)|^{p-2}(v(x)-v(y))}{|x-y|^{N+ps}} \, dydx
	\\
	\label{eq:frac:pointwise-weak}
	&=	\int_{\Omega} \xi(x) \mathcal{L}^s_p v(x) \,dx.
\end{align}
Thus, we get
$$
\langle F'(v),\xi\rangle = \int_{\Omega} \xi (\mathcal{L}^s_p v - h) \,dx 
\leq 
\|\mathcal{L}^s_p v - h\|_{p'} \|\xi\|_p 
\leq
C_{F,s} \|\mathcal{L}^s_p v - h\|_{p'} [\xi]_p,
$$
which implies \eqref{eq:nonlocal-main2}.

If there is no regularity $\mathcal{L}^s_p v \in L^{p'}(\Omega)$, then we argue in the following way (cf.\ \eqref{eq:thm1:proof:1}). 
For any $\xi \in \Wso$ and $\gamma^* \in Q_s^*$, we have 
\begin{align}
	\langle F'(v), \xi \rangle
	=
	\frac{1}{p}\langle D[v]_p^p, \xi \rangle 
	&- \iint_{\mathbb{R}^{2N}} \gamma^*(x,y) (\xi(x)-\xi(y)) \, \frac{dxdy}{|x-y|^{N+s}}
	\\
	\label{eq:nonloc:fin1}
	&+
	\iint_{\mathbb{R}^{2N}} \gamma^*(x,y) (\xi(x)-\xi(y)) \, \frac{dxdy}{|x-y|^{N+s}}
	 - \intO h \xi \,dx.
\end{align}	 
The first pair of terms on the right-hand side of \eqref{eq:nonloc:fin1} is estimated by the H\"older inequality and in view of \eqref{eq:Dupxi} and \eqref{eq:grad-div-nonlocal} as 
\begin{align}
	\frac{1}{p}\langle D[v]_p^p, \xi \rangle 
	&- \iint_{\mathbb{R}^{2N}} \gamma^*(x,y) (\xi(x)-\xi(y)) \, \frac{dxdy}{|x-y|^{N+s}}\\
	&=
	\iint_{\mathbb{R}^{2N}} \frac{|\nabla^{s}v|^{p-2}\nabla^{s}v  - \gamma^*}{|x-y|^{N/p'}} \, \frac{\xi(x)-\xi(y)}{|x-y|^{(N+sp)/p}} \, dxdy
	\\
	\label{eq:nonloc:fin2}
	&\leq
	\left(\mathcal{P}(|\nabla^{(s)} v|^{p-2} \nabla^{(s)} v - \gamma^*)\right)^{1/p'} [\xi]_p,
\end{align}
where at the last step the notations \eqref{eq:nonlocal-1} and \eqref{eq:nonloc:N} were also taken into account.
The second pair of terms on the right-hand side of \eqref{eq:nonloc:fin1} is estimated by virtue of Fubini's theorem and the skew symmetry of the function $\gamma^*$ as 
\begin{align}
	\iint_{\mathbb{R}^{2N}} \gamma^*(x,y) &(\xi(x)-\xi(y)) \, \frac{dxdy}{|x-y|^{N+s}}
	- 
	\intO h \xi \,dx\\
	&=
	2\int_{\mathbb{R}^{N}} \xi(x) 
	\int_{\mathbb{R}^{N}} \gamma^*(x,y)\,\frac{dxdy}{|x-y|^{N+s}}
	- \intO h \xi \,dx
	\\
	\label{eq:nonloc:fin3}
	&=
	\intO (\mathrm{div}^{(s)}{\gamma^*} - h) \xi \,dx 
	\leq
	C_{F,s}\|\mathrm{div}^{(s)}{\gamma^*} - h\|_{p'} [\xi]_p.
\end{align}
Combining \eqref{eq:nonloc:fin1}--\eqref{eq:nonloc:fin3}, we deduce \eqref{eq:nonlocal-main3}.
\noeqref{eq:nonloc:fin2}
\end{proof}

\begin{remark}
	If the solution $u$ satisfies $\mathcal{L}^s_p u \in L^{p'}(\Omega)$, then $|\nabla^{(s)} u|^{p-2}\nabla^{(s)} u \in Q_s^*$.
	In this case, the choice $\gamma^*(x,y) = |\nabla^{(s)} u|^{p-2}\nabla^{(s)} u$ makes the second term on the right-hand side of \eqref{eq:nonlocal-main3} vanish.
	Moreover, if $v \to u$ in $\Wso$, then the first term on the right-hand side of \eqref{eq:nonlocal-main3} tends to zero.
	Indeed, applying the estimates \eqref{eq:algebr-ineq1:anisotrop:p<2:2} and \eqref{eq:algebr-ineq1:anisotrop:p>2:2} (with $\mathcal{A}=I$), we obtain
	\begin{align}
		\mathcal{P}(|\nabla^{(s)} v|^{p-2}\nabla^{(s)} v - |\nabla^{(s)} u|^{p-2}\nabla^{(s)} u)
		\leq
		C(p,s,h,v,\Omega) \,
		[u-v]_p^{\min\{p,p'\}}.
	\end{align}
	The estimate \eqref{eq:nonlocal-main3} is therefore a consistent a posteriori estimate.
\end{remark}

\begin{remark}
	The assumption $\mathcal{L}^s_p v \in L^{p'}(\Omega)$ can be guaranteed if, e.g., $v$ is Lipschitz in $\mathbb{R}^N$ and $s < 1-1/p$. 
	In this case, we have the equality $\mathcal{L}^s_p v = (-\Delta)^s_p v$, where
	\begin{equation}\label{eq:nonlocal-0}
		(-\Delta)^s_p v(x) 
		= 
		2 \,
		\text{v.p.}
		\int_{\mathbb{R}^N}
		\frac{|v(x)-v(y)|^{p-2}(v(x)-v(y))}{|x-y|^{N+ps}} \, dy.
	\end{equation} 
	The equality \eqref{eq:nonlocal-0} is often used to define the fractional $p$-Laplacian.
	The absence of the principal value of the integral in the definitions of $\mathcal{L}^s_p$ and $\text{div}^{(s)}$ narrows their domain of definition, but allows to get easily \eqref{eq:frac:pointwise-weak} and \eqref{eq:nonloc:fin3}.
\end{remark}	
	
\begin{remark}
	The minimization problem for the functional $F$ admits a dual formulation, see \cite[Proposition~8.3]{BLP}.
	Thus, the result of Theorem~\ref{thm:nonlocal} might be developed using methods of the duality theory (see \cite{RepBook1}). 
	We also refer to the overviews in \cite{AG,BLN} devoted to the regularity and numerical implementations of problems of the type \eqref{eq:DEu} with nonlocal operators.
\end{remark}

	\appendix
	\section{Appendix}\label{sec:appendix}
	
	For convenience, we present a few necessary inequalities with explicit constants, see \cite{PY,SZ}.
	
	\begin{lemma}\label{lem:inequalities:anisotropic}
		Let $p>1$ and $\mathcal{A}=\mathcal{A}(x)$, $x \in \Omega$, be a measurable symmetric matrix satisfying the assumptions \eqref{eq:anisotropic:A} and \eqref{eq:anisotropin:assumption}.
		The following assertions hold:
		\begin{enumerate}[label={\rm(\roman*)}]
		\item\label{lem:inequalities:anisotropic:p<2} 
		Let $p \leq 2$.
		Then, for any $a,b \in \mathbb{R}^N$ and $x \in \Omega$,
		\begin{align}
			\label{eq:algebr-ineq1:anisotrop:p<2:1}
			&(|b|^{p-2}\mathcal{A}(x)b - |a|^{p-2}\mathcal{A}(x)a) \cdot (b-a) 
			\geq 
			\frac{\lambda(x) \delta_\Omega(p,\mathcal{A})}{2} \frac{|b-a|^2}{(|b|+|a|)^{2-p}},\\
			\label{eq:algebr-ineq1:anisotrop:p<2:2}
			&||b|^{p-2} \mathcal{A}(x) b - |a|^{p-2} \mathcal{A}(x) a| 
			\leq
			\Lambda(x) \frac{3-p}{p-1} \,2^{p-1} \, |b-a|^{p-1},
		\end{align}	
		where the right-hand side of \eqref{eq:algebr-ineq1:anisotrop:p<2:1} is assumed to be zero for $a=b=0$.
	
		\item\label{lem:inequalities:anisotropic:p>2} 
		Let $p \geq 2$.
		Then, for any $a,b \in \mathbb{R}^N$ and $x \in \Omega$,
		\begin{align}
		\label{eq:algebr-ineq1:anisotrop:p>2:1}
		&(|b|^{p-2}\mathcal{A}(x)b - |a|^{p-2}\mathcal{A}(x)a) \cdot (b-a) 
		\geq \frac{\lambda(x) \, \delta_\Omega(p,\mathcal{A})}{8 (p-1)} \, |b-a|^p,\\
		\label{eq:algebr-ineq1:anisotrop:p>2:2}
		&||b|^{p-2} \mathcal{A}(x) b - |a|^{p-2} \mathcal{A}(x) a| 
		\leq \Lambda(x) (p-1) (|a|+|b|)^{p-2} |b-a|.
		\end{align}
		\end{enumerate}
	\end{lemma}
	\begin{proof}
		\ref{lem:inequalities:anisotropic:p<2} 
		Let $p \leq 2$.
		The inequality \eqref{eq:algebr-ineq1:anisotrop:p<2:1} follows from \cite[(5.5)]{SZ}:
		$$
		(|b|^{p-2}\mathcal{A}(x)b - |a|^{p-2}\mathcal{A}(x)a) \cdot (b-a)
		\geq 
		\frac{\lambda(x)}{2} \left(p - (2-p)\mu(\mathcal{A})\right) |b-a|^2 
		\int_0^1 |(1-s)a+sb|^{p-2} \,ds,
		$$
		and from the estimate $\int_0^1 |(1-s)a+sb|^{p-2} \,ds \geq (|a|+|b|)^{p-2}$ and the assumption \eqref{eq:anisotropin:assumption}.
		
		The inequality \eqref{eq:algebr-ineq1:anisotrop:p<2:2} follows from \cite[(5.4)]{SZ}:
		$$
		||b|^{p-2} \mathcal{A}(x) b - |a|^{p-2} \mathcal{A}(x) a| 
		\leq
		\Lambda(x) (3-p) |b-a| \int_0^1 |(1-s)a+sb|^{p-2} \,ds,
		$$
		and from \cite[(5.6)]{SZ}: $\int_0^1 |(1-s)a+sb|^{p-2} \,ds \leq C |b-a|^{p-2}$, where 
		$$
		C = \max\{(p-1)^{-1} 2^{p-1}, 2^{p-2}\} = (p-1)^{-1} 2^{p-1}.
		$$

		\ref{lem:inequalities:anisotropic:p>2} 
		Let $p \geq 2$. 
		The last inequality in \cite[p.~1809]{SZ} reads as 
		\begin{equation}\label{eq:anis-0x}
			(|b|^{p-2}\mathcal{A}(x)b - |a|^{p-2}\mathcal{A}(x)a) \cdot (b-a) 
			\geq
			\frac{\lambda(x)}{2} 
			\left(p - (p-2)\mu(\mathcal{A})\right) |b-a|^2 
			\int_0^1 |(1-s)a+sb|^{p-2} \,ds.
		\end{equation}
		The following inequality was also obtained therein: 
		\begin{equation}\label{eq:anis-1x1}
		\int_0^1 |(1-s)a+sb|^{p-2} \,ds 
		\geq
		\frac{|a|^{p-2} + |b|^{p-2}}{2(p-1)}.
		\end{equation}
		Using $(1+|t|)^{p-2} \leq 2 (1+|t|^{p-2})$, we arrive at
		\begin{equation}\label{eq:anis-1x2}
		\int_0^1 |(1-s)a+sb|^{p-2} \,ds 
		\geq
		\frac{(|a| + |b|)^{p-2}}{4(p-1)}
		\geq
		\frac{|b-a|^{p-2}}{4(p-1)}.
		\end{equation}
		Combining \eqref{eq:anis-0x} with \eqref{eq:anis-1x2}, we deduce \eqref{eq:algebr-ineq1:anisotrop:p>2:1}.
		
		The inequality \eqref{eq:algebr-ineq1:anisotrop:p>2:2} follows from the last inequality in \cite[p.~1809]{SZ}:
		$$
		||b|^{p-2} \mathcal{A}(x) b - |a|^{p-2} \mathcal{A}(x) a|
		\leq
		\Lambda(x) (p-1) |b-a| \int_0^1 |(1-s)a + sb|^{p-2} \,ds,
		$$
		by using the estimate $\int_0^1 |(1-s)a+sb|^{p-2} \,ds \leq (|a|+|b|)^{p-2}$.
	\end{proof}

	The obtained estimates, in particular, allow to prove the following result. 
	\begin{lemma}\label{lem:flows-anisotrop}
	Let $p>1$ and $\mathcal{A}=\mathcal{A}(x)$, $x \in \Omega$, be a measurable symmetric matrix satisfying the assumptions \eqref{eq:anisotropic:A} and \eqref{eq:anisotropin:assumption}.
	Then for any $v, w \in W^{1,p}(\Omega)$ the following assertions hold:
	\begin{enumerate}[label={\rm(\roman*)}]
		\item\label{lem:flows-anisotrop:1} If $p \leq 2$, then 
		\begin{equation}\label{eq:anisotrop:flow1}
			\||\nabla w|^{p-2} \mathcal{A}\nabla w - |\nabla v|^{p-2} \mathcal{A}\nabla v\|_{p'}
			\leq
			\sup_{x \in \Omega} \Lambda(x) \frac{3-p}{p-1} \,2^{p-1}
			\|\nabla (w-v)\|_p^{p-1}.
		\end{equation}
		\item\label{lem:flows-anisotrop:2} If $p \geq 2$, then 
		\begin{align}
			&\||\nabla w|^{p-2} \mathcal{A} \nabla w - |\nabla v|^{p-2} \mathcal{A} \nabla v\|_{p'}
			\\
			\label{eq:anisotrop:flow2}
			&\leq 
			\sup_{x \in \Omega} \Lambda(x) \, (p-1) 2^{\frac{p-2}{p}}
			\left( \|\nabla w\|_p^p + \|\nabla v\|_p^p\right)^{\frac{p-2}{p}}
			\|\nabla (w-v)\|_p.
		\end{align}
	\end{enumerate}
\end{lemma}
\begin{proof}
	In the case $p \leq 2$, the required estimate follows from \eqref{eq:algebr-ineq1:anisotrop:p<2:2}.
	In the case $p \geq 2$, applying the inequality \eqref{eq:algebr-ineq1:anisotrop:p>2:2}, we have
	\begin{align}
		&\intO ||\nabla w|^{p-2} \mathcal{A}(x) \nabla w - |\nabla v|^{p-2} \mathcal{A}(x)\nabla v|^{p'} \,dx
		\\
		&\leq
		\sup_{x \in \Omega} (\Lambda(x))^{p'} (p-1)^{p'} \intO (|\nabla w| + |\nabla v|)^{(p-2)p'} |\nabla (w-v)|^{p'} \,dx\\
		&\leq
		\sup_{x \in \Omega} (\Lambda(x))^{p'} (p-1)^{p'}
		\left(\intO (|\nabla w| + |\nabla v|)^{p} \,dx\right)^{\frac{p-2}{p-1}}
		\left( \intO |\nabla (w-v)|^{p} \,dx \,dx\right)^{\frac{1}{p-1}}\\
		&\leq 
		\sup_{x \in \Omega} (\Lambda(x))^{p'} (p-1)^{p'} 2^{\frac{p-2}{p-1}}
		\left(\intO |\nabla w|^p \,dx + \intO |\nabla v|^{p}  \,dx\right)^{\frac{p-2}{p-1}}
		\left( \intO |\nabla (w-v)|^{p} \,dx \,dx\right)^{\frac{1}{p-1}},
	\end{align}
	which implies \eqref{eq:anisotrop:flow2}.
\end{proof}

Let us formulate an analogue of Lemma~\ref{lem:flows-anisotrop} in the special case $\mathcal{A} = I$, with better constants and with inverse inequalities.
\begin{lemma}\label{lem:flows1}
	Let $p>1$ and $v, w \in W^{1,p}(\Omega)$.
	The following assertions hold:
	\begin{enumerate}[label={\rm(\roman*)}]
		\item\label{lem:flows1:1} If $p \leq 2$, then 
		\begin{align}
			\label{eq:thm12:flow1}
			\||\nabla w|^{p-2} \nabla w - |\nabla v|^{p-2} \nabla v\|_{p'}
			&\leq
			2^{2-p}
			\|\nabla (w-v)\|_p^{p-1},\\
			\label{eq:thm12:flow1inv}
			(p-1) 2^{\frac{p-2}{p}}
			\left( \|\nabla w\|_p^p + \|\nabla v\|_p^p\right)^{\frac{p-2}{p}}
			\|\nabla (w-v)\|_p
			&\leq
			\||\nabla w|^{p-2} \nabla w - |\nabla v|^{p-2} \nabla v\|_{p'}.
		\end{align}
		\item\label{lem:flows1:2} If $p \geq 2$, then 
		\begin{align}
			\label{eq:thm12:flow2}
			\||\nabla w|^{p-2} \nabla w - |\nabla v|^{p-2} \nabla v\|_{p'}
			&\leq 
			(p-1) 2^{\frac{p-2}{p}}
			\left( \|\nabla w\|_p^p + \|\nabla v\|_p^p\right)^{\frac{p-2}{p}}
			\|\nabla (w-v)\|_p,~~~\\
			\label{eq:thm12:flow2inv}
			2^{2-p}\|\nabla (w-v)\|_p^{p-1} 
			&\leq 
			\||\nabla w|^{p-2} \nabla w - |\nabla v|^{p-2} \nabla v\|_{p'}.
		\end{align}
	\end{enumerate}
\end{lemma}
\begin{proof}
	In the case $p \leq 2$, the estimate \eqref{eq:thm12:flow1} is a direct consequence of the algebraic inequality \cite[Chapter~12, (VIII)]{Lind}:
	\begin{equation}\label{eq:algebr4-0}
		||b|^{p-2} b - |a|^{p-2} a| \leq 2^{2-p} |b-a|^{p-1}, 
		\quad a,b \in \mathbb{R}^N, \quad p \leq 2.
	\end{equation}
	The estimate \eqref{eq:thm12:flow1inv} is obtained by arguments similar to \eqref{eq:thm2:p<2}. Namely, by virtue of the H\"older inequality and \eqref{eq:algebr-ineq2}, we obtain
	\begin{align}
		&\intO |\nabla (w-v)|^p \,dx
		=
		\intO \left(\frac{|\nabla (w-v)|^{2}}{(|\nabla u| + |\nabla v|)^{2-p}}\right)^{\frac{p}{2}} (|\nabla w| + |\nabla v|)^{\frac{p(2-p)}{2}} \,dx\\
		&\leq
		\left(\intO \frac{|\nabla (w-v)|^2}{(|\nabla w| + |\nabla v|)^{2-p}} \,dx \right)^{\frac{p}{2}}
		\left(\intO (|\nabla w| + |\nabla v|)^{p} \, dx\right)^{\frac{2-p}{2}}\\
		&\leq
		\left(\frac{1}{p-1} \intO (|\nabla w|^{p-2} \nabla u - |\nabla v|^{p-2} \nabla v) \cdot (\nabla w - \nabla v) \,dx\right)^{\frac{p}{2}}
		\left( 2\intO |\nabla w|^p \,dx + 2 \intO |\nabla v|^{p} \, dx\right)^{\frac{2-p}{2}}.
	\end{align}
	Thus, 
	\begin{align}
		(p-1) 2^\frac{p-2}{p} 
		(\|\nabla w\|_p^p+\|\nabla v\|_p^p)^\frac{p-2}{p}
		\|\nabla (w-v)\|_p^2 
		&\leq
		 \intO (|\nabla w|^{p-2} \nabla w - |\nabla v|^{p-2} \nabla v) \cdot (\nabla w - \nabla v) \,dx
		 \\
		&\leq
		\||\nabla w|^{p-2} \nabla w - |\nabla v|^{p-2} \nabla v\|_{p'} \|\nabla (w-v)\|_p,
	\end{align}
	which yields \eqref{eq:thm12:flow1inv}.
	
	In the case $p \geq 2$, a simple estimate of the equality \cite[Chapter~12, (IV)]{Lind} gives
	\begin{equation}\label{eq:algebr4}
		||b|^{p-2} b - |a|^{p-2} a| \leq (p-1) (|a|+|b|)^{p-2} |b-a|, \quad a,b \in \mathbb{R}^N, \quad p \geq 2.
	\end{equation}
	Using \eqref{eq:algebr4}, the H\"older inequality, and  $(1+|t|)^p \leq 2 (1+|t|^p)$, we obtain
	\begin{align}
		&\intO ||\nabla w|^{p-2} \nabla w - |\nabla v|^{p-2} \nabla v|^{p'} \,dx
		\leq
		(p-1)^{p'} \intO (|\nabla w| + |\nabla v|)^{(p-2)p'} |\nabla (w-v)|^{p'} \,dx\\
		&\leq
		(p-1)^{p'}
		\left(\intO (|\nabla w| + |\nabla v|)^{p} \,dx\right)^{\frac{p-2}{p-1}}
		\left( \intO |\nabla (w-v)|^{p} \,dx \,dx\right)^{\frac{1}{p-1}}\\
		\label{eq:lem:flows1:2}
		&\leq 
		(p-1)^{p'} 2^{\frac{p-2}{p-1}}
		\left(\intO |\nabla w|^p \,dx + \intO |\nabla v|^{p}  \,dx\right)^{\frac{p-2}{p-1}}
		\left( \intO |\nabla (w-v)|^{p} \,dx \,dx\right)^{\frac{1}{p-1}},
	\end{align}
	which implies \eqref{eq:thm12:flow2}.
	
	The estimate \eqref{eq:thm12:flow2inv} follows from the inequality \eqref{eq:algebr-ineq1}.
\end{proof}

\medskip
\noindent
\textbf{Acknowledgement.}
{The authors express their gratitude to D.E.~Apushkinskaya and S.I.~Repin for useful discussions and comments, which helped to strengthen some results and improve the final version of the text.}


\addcontentsline{toc}{section}{\refname}
	\small
	\begin{thebibliography}{99}
				
		\bibitem{AG}
		Ainsworth, M., \& Glusa, C. (2017). Aspects of an adaptive finite element method for the fractional Laplacian: a priori and a posteriori error estimates, efficient implementation and multigrid solver. Computer Methods in Applied Mechanics and Engineering, 327, 4-35.
		\href{https://doi.org/10.1016/j.cma.2017.08.019}{doi:10.1016/j.cma.2017.08.019}
		
		\bibitem{AO}
		Ainsworth, M., \& Oden, J. T. (1997). A posteriori error estimation in finite element analysis. Computer Methods in Applied Mechanics and Engineering, 142(1-2), 1-88.
		\href{https://doi.org/10.1016/S0045-7825(96)01107-3}{doi:10.1016/S0045-7825(96)01107-3}
		
		\bibitem{antonini}
		Antonini, C. A., Ciraolo, G., \& Farina, A. (2023). Interior regularity results for inhomogeneous anisotropic quasilinear equations. Mathematische Annalen, 387(3), 1745-1776.
		\href{https://doi.org/10.1007/s00208-022-02500-x}{doi:10.1007/s00208-022-02500-x}
		
		\bibitem{ANR}
		Apushkinskaya, D. E., Novikova, A. A., \& Repin, S. I. (2024). A posteriori error estimates for approximate solutions to the obstacle problem for the $p$-Laplacian. Differential Equations, 60(10), 1476-1490.
		\href{https://doi.org/10.1134/S0012266124100100}{doi:10.1134/S0012266124100100}
		
		\bibitem{AR}
		Apushkinskaya, D. E., \& Repin, S. I. (2020). Biharmonic obstacle problem: guaranteed and computable error bounds for approximate solutions. Computational Mathematics and Mathematical Physics, 60, 1823-1838.
		\href{https://doi.org/10.1134/S0965542520110032}{doi:10.1134/S0965542520110032}
		
		\bibitem{BL}
		Barrett, J. W., \& Liu, W. B. (1993). Finite element approximation of the $p$-Laplacian. Mathematics of Computation, 61(204), 523-537.
		\href{https://doi.org/10.1090/S0025-5718-1993-1192966-4}{doi:10.1090/S0025-5718-1993-1192966-4}
		
		\bibitem{BD}
		Benedikt, J., \& Dr\'abek, P. (2012). Estimates of the principal eigenvalue of the $p$-Laplacian. Journal of Mathematical Analysis and Applications, 393(1), 311-315.
		\href{https://doi.org/10.1016/j.jmaa.2012.03.054}{doi:10.1016/j.jmaa.2012.03.054}

		\bibitem{Besov}
		Besov, K. O. (2023). Integral identity and estimate of the deviation of approximate solutions of a biharmonic obstacle problem. Computational Mathematics and Mathematical Physics, 63(3), 333-336.
		\href{https://doi.org/10.1134/S096554252303003X}{doi:10.1134/S096554252303003X}

		\bibitem{BR}
		Bildhauer, M., \& Repin, S. I. (2007). Estimates of the deviation from the minimizer for variational problems with power growth functionals. Journal of Mathematical Sciences, 143, 2845-2856.
		\href{https://doi.org/10.1007/s10958-007-0170-x}{doi:10.1007/s10958-007-0170-x}

		\bibitem{BLN}
		Borthagaray, J. P., Li, W., \& Nochetto, R. H. (2023). Fractional Elliptic Problems on Lipschitz Domains: Regularity and Approximation. A$^3$N$^2$M: Approximation, Applications, and Analysis of Nonlocal, Nonlinear Models: Proceedings of the 50th John H. Barrett Memorial Lectures, 165, 27.
		\href{https://doi.org/10.1007/978-3-031-34089-5_2}{doi:10.1007/978-3-031-34089-5\_2}
		
		\bibitem{BLP}
		Brasco, L., Lindgren, E., \& Parini, E. (2014). The fractional Cheeger problem. Interfaces and Free Boundaries, 16(3), 419-458.
		\href{https://doi.org/10.4171/IFB/325}{doi:10.4171/IFB/325}
		
		\bibitem{CF}
		Caffarelli, L. A., \& Friedman, A. (1979). The obstacle problem for the biharmonic operator. Annali della Scuola Normale Superiore di Pisa-Classe di Scienze, 6(1), 151-184.
		\url{http://eudml.org/doc/83803}
		
		\bibitem{CK}
		Carstensen, C., \& Klose, R. (2003). A posteriori finite element error control for the $p$-Laplace problem. SIAM Journal on Scientific Computing, 25(3), 792-814.
		\href{https://doi.org/10.1137/S1064827502416617}{doi:10.1137/S1064827502416617}
		
		\bibitem{CM}
		Carstensen, C., \& M\"uller, S. (2002). Local stress regularity in scalar nonconvex variational problems. SIAM Journal on Mathematical Analysis, 34(2), 495-509.
		\href{https://doi.org/10.1137/S0036141001396436}{doi:10.1137/S0036141001396436}
		
		\bibitem{caselli}
		Caselli, M., Gentile, A., \& Giova, R. (2020). Regularity results for solutions to obstacle problems with Sobolev coefficients. Journal of Differential Equations, 269(10), 8308-8330.
		\href{https://doi.org/10.1016/j.jde.2020.06.015}{doi:10.1016/j.jde.2020.06.015}
		
		\bibitem{chensong}
		Chen, Z. Q., \& Song, R. (2005). Two-sided eigenvalue estimates for subordinate processes in domains. Journal of Functional Analysis, 226(1), 90-113.
		\href{https://doi.org/10.1016/j.jfa.2005.05.004}{doi:10.1016/j.jfa.2005.05.004}
		
		\bibitem{ciacni-mazja}
		Cianchi, A., \& Maz'ya, V. G. (2018). Second-order two-sided estimates in nonlinear elliptic problems. Archive for Rational Mechanics and Analysis, 229, 569-599.
		\href{https://doi.org/10.1007/s00205-018-1223-7}{doi:10.1007/s00205-018-1223-7}
		
		\bibitem{CP}
		Colasuonno, F., \& Pucci, P. (2011). Multiplicity of solutions for $p(x)$-polyharmonic elliptic Kirchhoff equations. Nonlinear Analysis: Theory, Methods \& Applications, 74(17), 5962-5974.
		\href{https://doi.org/10.1016/j.na.2011.05.073}{doi:10.1016/j.na.2011.05.073}
		
		\bibitem{diben}
		DiBenedetto, E. (1983). $C^{1+\alpha}$ local regularity of weak solutions of degenerate elliptic equations. 
		Nonlinear Analysis: Theory, Methods \& Applications, 7(8), 827-850.
		\href{http://dx.doi.org/10.1016/0362-546X(83)90061-5}{doi:10.1016/0362-546X(83)90061-5}
	
		\bibitem{FR}
		Farwig, R., \& Rosteck, V. (2016). Note on Friedrichs' inequality in $N$-star-shaped domains. Journal of Mathematical Analysis and Applications, 435(2), 1514-1524.
		\href{http://dx.doi.org/10.1016/j.jmaa.2015.10.046}{doi:10.1016/j.jmaa.2015.10.046}
		
		\bibitem{Fre}
		Frehse, J. (1971). Zum differenzierbarkeitsproblem bei variationsungleichungen h\"oherer ordnung. Abhandlungen aus dem Mathematischen Seminar der Universit\"at Hamburg, 36(1), 140-149.
		\href{http://dx.doi.org/10.1007/BF02995917}{doi:10.1007/BF02995917}

		\bibitem{Frolov1}
		Frolov, M. E. (2004). On efficiency of the dual majorant method for the quality estimation of approximate solutions of fourth-order elliptic boundary value problems. Russian Journal of Numerical Analysis \& Mathematical Modelling, 19(5), 407-418.
		\href{http://dx.doi.org/10.1515/1569398042395970}{doi:10.1515/1569398042395970}

		\bibitem{fuchs}
		Fuchs, M. (1990). H\"older continuity of the gradient for degenerate variational inequalities. Nonlinear Analysis: Theory, Methods \& Applications, 15(1), 85-100.
		\href{http://dx.doi.org/10.1016/0362-546X(90)90016-A}{doi:10.1016/0362-546X(90)90016-A}
		
		\bibitem{GGS}
		Gazzola, F., Grunau, H. C., \& Sweers, G. (2010). Polyharmonic boundary value problems: positivity preserving and nonlinear higher order elliptic equations in bounded domains. Springer Science \& Business Media.
		\href{http://dx.doi.org/10.1007/978-3-642-12245-3}{doi:10.1007/978-3-642-12245-3}
		
		\bibitem{grisvard}
		Grisvard, P. (1985). Elliptic problems in nonsmooth domains. Pitman Publishing. 
		\href{http://dx.doi.org/10.1137/1.9781611972030}{doi:10.1137/1.9781611972030}
		
		\bibitem{guarnotta}
		Guarnotta, U., \& Mosconi, S. (2023). A general notion of uniform ellipticity and the regularity of the stress field for elliptic equations in divergence form. Analysis \& PDE, 16(8), 1955-1988. 
		\href{http://dx.doi.org/10.2140/apde.2023.16.1955}{doi:10.2140/apde.2023.16.1955}
	
		\bibitem{hepp}
		Hepp, S., \& Kassmann, M. (2024). The divergence theorem and nonlocal counterparts. Bulletin of the London Mathematical Society, 56(2), 711-733.
		\href{http://dx.doi.org/10.1112/blms.12960}{doi:10.1112/blms.12960}
	
		\bibitem{lieberman}
		Lieberman, G. M. (1988). Boundary regularity for solutions of degenerate elliptic equations. Nonlinear Analysis: Theory, Methods \& Applications, 12(11), 1203-1219.
		\href{http://dx.doi.org/10.1016/0362-546X(88)90053-3}{doi:10.1016/0362-546X(88)90053-3}

		\bibitem{LL}
		Lindgren, E., \& Lindqvist, P. (2014). Fractional eigenvalues. Calculus of Variations and Partial Differential Equations, 49(1), 795-826.
		\href{http://dx.doi.org/10.1007/s00526-013-0600-1}{doi:10.1007/s00526-013-0600-1}

		\bibitem{Lind}
		Lindqvist, P. (2019). Notes on the stationary $p$-Laplace equation. Berlin: Springer International Publishing.
		\href{http://dx.doi.org/10.1007/978-3-030-14501-9}{doi:10.1007/978-3-030-14501-9}
		
		\bibitem{lions}
		Lions, J. L. (1969). Quelques m\'ethodes de r\'esolution des probl\`emes aux limites non-lin\'eaires. Dunod.
		
		\bibitem{LY}
		Liu, W., \& Yan, N. (2001). Quasi-norm local error estimators for $p$-Laplacian. SIAM Journal on Numerical Analysis, 39(1), 100-127.
		\href{http://dx.doi.org/10.1137/S0036142999351613}{doi:10.1137/S0036142999351613}
		
		\bibitem{NR}
		Neittaan\"aki, P., \& Repin, S. I. (2001). A posteriori error estimates for boundary-value problems related to the biharmonic operator. Journal of Numerical Mathematics, 9(2), 157-178.
		\href{http://dx.doi.org/10.1515/JNMA.2001.157}{doi:10.1515/JNMA.2001.157}
		
		\bibitem{P1}
		Pastukhova, S. E. (2020). A posteriori estimates for the accuracy of approximations in variational problems with power functionals. Journal of Mathematical Sciences, 244(3), 509-523.
		\href{http://dx.doi.org/10.1007/s10958-019-04631-0}{doi:10.1007/s10958-019-04631-0}
		
		\bibitem{PYa2}
		Pastukhova, S. E., \& Yakubovich, D. A. (2019). Galerkin approximations for the Dirichlet problem with the $p(x)$-Laplacian. 
		Sbornik: Mathematics, 210(1), 145–164.
		\href{https://doi.org/10.1070/SM9019}{doi:10.1070/SM9019}
		
		\bibitem{PY}
		Pastukhova, S. E., \& Yakubovich, D. A. (2019). Galerkin approximations in problems with anisotropic $p(\cdot)$-Laplacian. Applicable Analysis, 98(1-2), 345-361.
		\href{http://dx.doi.org/10.1080/00036811.2018.1451641}{doi:10.1080/00036811.2018.1451641}
		
		\bibitem{Rep2000-2}
		Repin, S. (2000). A posteriori error estimation for variational problems with uniformly convex functionals. Mathematics of Computation, 69(230), 481-500.
		\href{http://dx.doi.org/10.1090/S0025-5718-99-01190-4}{doi:10.1090/S0025-5718-99-01190-4}

		\bibitem{Rep2000-1}
		Repin, S. I. (2000). A posteriori error estimation for nonlinear variational problems by duality theory. Journal of Mathematical Sciences, 99, 927-935.
		\href{http://dx.doi.org/10.1007/BF02673600}{doi:10.1007/BF02673600}

		\bibitem{RepBook1}
		Repin, S. (2008). A posteriori estimates for partial differential equations. Walter de Gruyter.
		\href{http://dx.doi.org/10.1515/9783110203042}{doi:10.1515/9783110203042}
		
		\bibitem{RepBook2}
		Repin, S., \& Sauter, S. A. (2020). Accuracy of mathematical models: Dimension reduction, homogenization, and simplification. EMS Tracts in Mathematics, 33.
		\href{http://dx.doi.org/10.4171/206}{doi:10.4171/206}
		
		\bibitem{Rep2023}
		Repin, S. I. (2023). A posteriori identities for measures of deviation from exact solutions of nonlinear boundary value problems. Computational Mathematics and Mathematical Physics, 63(6), 934-956.
		\href{https://doi.org/10.1134/S0965542523060155}{doi:10.1134/S0965542523060155}

		\bibitem{rodrig}
		Rodrigues, J. F. (2005). Stability remarks to the obstacle problem for $p$-Laplacian type equations. Calculus of Variations and Partial Differential Equations, 23(1), 51-65.
		\href{http://dx.doi.org/10.1007/s00526-004-0288-3}{doi:10.1007/s00526-004-0288-3}
		
		\bibitem{SZ}
		Surnachev, M. D., \& Zhikov, V. V. (2013). On existence and uniqueness classes for the Cauchy problem for parabolic equations of the $p$-Laplace type. Communications on Pure \& Applied Analysis, 12(4), 1783-1812.
		\href{http://dx.doi.org/10.3934/cpaa.2013.12.1783}{doi:10.3934/cpaa.2013.12.1783}
		
		\bibitem{Verf}
		Verf\"urth, R. (2013). A posteriori error estimation techniques for finite element methods. Oxford University Press.
		\href{http://dx.doi.org/10.1093/acprof:oso/9780199679423.001.0001}{doi:10.1093/acprof:oso/9780199679423.001.0001}
		
		\bibitem{YY}
		Yolcu, S. Y., \& Yolcu, T. (2014). Eigenvalue bounds for the poly-harmonic operators. Illinois Journal of Mathematics, 58(3), 847-865.
		\href{http://dx.doi.org/10.1215/ijm/1441790392}{doi:10.1215/ijm/1441790392}

		\bibitem{ZY}
		Zhikov, V., \& Yakubovich, D. (2016). Galerkin approximations in problems with $p$-Laplacian. Journal of Mathematical Sciences, 219(1).
		\href{http://dx.doi.org/10.1007/s10958-016-3086-5}{doi:10.1007/s10958-016-3086-5}

	\end{thebibliography}
\end{document}